\documentclass[reqno]{amsart}

\usepackage{amssymb,amsmath}
\usepackage{euscript}
\usepackage{graphicx}
\usepackage{tikz}
\usepackage{tikz-cd}
\usepackage{hyperref}
\usepackage{mathrsfs}
\usepackage{dsfont}
\usepackage{mathtools}
\usepackage{microtype}

\newtheorem{thm}{Theorem}[section]
\newtheorem{cor}[thm]{Corollary}
\newtheorem{lem}[thm]{Lemma}
\newtheorem{prop}[thm]{Proposition}
\newtheorem{defin}[thm]{Definition}
\newtheorem{def-lem}[thm]{Definition-Lemma}
\newtheorem{conj}[thm]{Conjecture}

\theoremstyle{remark}

\newtheorem{rem}[thm]{Remark}

\newtheorem{example}[thm]{Example}

\newtheorem{ques}{Question}

\newtheorem{notation}[thm]{Notation}

\numberwithin{equation}{section}

\newcommand{\bbC}{\mathbb{C}}

\newcommand{\bbF}{\mathbb{F}}
\newcommand{\bbI}{\mathbb{I}}

\newcommand{\bbM}{\mathbb{M}}

\newcommand{\bbR}{\mathbb{R}}

\newcommand{\bbS}{\mathbb{S}}
\newcommand{\bbU}{\mathbb{U}}
\newcommand{\bbX}{\mathbb{X}}

\newcommand{\bbZ}{\mathbb{Z}}

\newcommand{\scrF}{\mathscr{F}}
\newcommand{\scrM}{\mathscr{M}}

\newcommand{\scrU}{\mathscr{U}}

\newcommand{\calA}{\mathcal{A}}

\newcommand{\calL}{\mathcal{L}}

\newcommand{\calP}{\mathcal{P}}

\newcommand{\calV}{\mathcal{V}}

\newcommand{\frakc}{\mathfrak{c}}
\newcommand{\frakF}{\mathfrak{F}}

\newcommand{\frakM}{\mathfrak{M}}

\newcommand{\frako}{\mathfrak{o}}
\newcommand{\frakR}{\mathfrak{R}}

\newcommand{\frakV}{\mathfrak{V}}
\newcommand{\frakX}{\mathfrak{X}}

\DeclareMathOperator{\codim}{codim}

\DeclareMathOperator{\crit}{Crit}
\DeclareMathOperator{\grad}{grad}

\DeclareMathOperator{\ind}{ind}

\newcommand{\eval}{\mathrm{eval}}
\newcommand{\flow}{\mathrm{Flow}}
\newcommand{\fr}{\mathrm{fr}}

\newcommand{\spectra}{\mathrm{Sp}}

\newcommand{\std}{\mathrm{std}}

\newcommand{\unit}{\mathds{1}}

\newcommand{\abs}[1]{\lvert#1\rvert}

\title[Spectral Viterbo isomorphism: complex-oriented versus framed]{Spectral Viterbo isomorphism: \\ complex-oriented versus framed}
\author{Kenneth Blakey}
\address{Department of Mathematics, MIT, 182 Memorial Drive, Cambridge, MA 02139, U.S.A.} \email{kblakey@mit.edu}

\begin{document}

\begin{abstract}
The Viterbo isomorphism relates the symplectic cohomology of a cotangent bundle to the homology of the free loop space of its base. We lift this to a relation of modules over (1) the complex bordism spectrum MU and (2) the sphere spectrum $\bbS$. In particular, by a result of Porcelli and the present author \cite{BP26}, it is not the case that, in general, the $\bbS$-level statement recovers the MU-level statement after base-change -- even in the case the base is spin.
\end{abstract}

\maketitle
\tableofcontents

\section{Introduction}

\subsection{Background}
Let $Q$ be a closed smooth manifold and $\calL Q$ its free loop space. Classically, the Viterbo isomorphism relates the symplectic cohomology of $T^*Q$ to the homology of $\calL Q$: 
    \begin{equation}\label{eqn:viterbocohomology}
    SH^{-*}(T^*Q;\bbZ)\cong H_*(\calL Q;\eta)\;\;{\rm or}\;\;SH^{-*}(T^*Q;\eta^{-1})\cong H_*(\calL Q;\bbZ),
    \end{equation}
where $\eta$ is a local system on $\calL Q$ which is trivial when $Q$ is spin, see \eqref{eqn:localsystem}. Some history is as follows. The relation \eqref{eqn:viterbocohomology} was proven, with $\bbZ/2$-coefficients on both sides, in \cite{AS06,SW06,Vit98}, and was expected to also hold with $\bbZ$-coefficients (perhaps after twisting either side by the orientation local system in the non-oriented case). The fact that the latter statement is not true, in the case that $Q$ is non-spin, was first noticed by Kragh \cite{Kra18} and verified by Seidel \cite{Sei10} in an unpublished note computing the case of $T^*\bbC P^2$. Eventually, a complete proof of \eqref{eqn:viterbocohomology} appeared in Abouzaid's monograph \cite{Abo14}.

Traditionally, symplectic cohomology is constructed via a cochain complex over a classical ring, e.g. $\bbZ$ or $\bbC$. Recently, there has been interest in constructing Floer-theoretic invariants over more exotic ring spectra, and this is the subject of Floer homotopy theory. The most general version of spectral symplectic cohomology which canonically exists on a graded\footnote{I.e., satisfying that the first Chern class is 2-torsion. In particular, $T^*Q$ is graded.} Liouville manifold is a module $\frakF^{\rm MU}$ over the complex bordism spectrum MU; this is because the moduli spaces of Floer trajectories can be equipped with canonical stable complex structures. Meanwhile, when a Liouville manifold admits a stable $\bbR$-polarization $\Lambda$, it admits a theory of spectral symplectic cohomology as a module $\frakF^{\bbS,\Lambda}$ over the sphere spectrum $\bbS$ since the Floer moduli spaces can be equipped with stable framings; an example of such a Liouville manifold is $T^*Q$ with its ``standard'' stable $\bbR$-polarization: 
    \begin{equation}
    T(T^*Q)\cong\Lambda_\std\otimes_\bbR\underline{\bbC},\;\;\Lambda_\std\equiv\pi^*TQ,\;\;\pi:T^*Q\to Q.
    \end{equation}
Here, an underline denotes a trivial vector bundle with the notated fiber, e.g. $\underline{\bbC}\equiv T^*Q\times\bbC\to T^*Q$. We would like to emphasize that $\frakF^{\bbS,\Lambda}$ really does depend on $\Lambda$, cf. for example \cite{BB25}. It is natural to wonder whether the base-change of an $\bbS$-Floer homotopy type recovers the canonical MU-Floer homotopy type. An obstruction to the canonical MU-Floer homotopy type being the base-change of \emph{any} $\bbS$-module was constructed in work of Porcelli and the present author \cite{BP26}; moreover, it was computed that the canonical MU-Floer homotopy type of $T^*\bbC P^n$, $n\in\bbZ$ at least 3 and odd, cannot possibly be base-changed from an $\bbS$-module.

Circling back to Viterbo's isomorphism, various special cases of the $\bbS$-Viterbo isomorphism exist in the literature: cf. \cite{Coh10} for when the base is spin, cf. \cite{CK23} for the $S^1$-equivariant case when the base is stably framed, and cf. \cite{ADP24} for the open-string analogue. The general (i.e., non-oriented/non-spin) $\bbS$-Viterbo isomorphism was proven in an unpublished note written by the present author. However, the computation in \cite{BP26} begs the question ``How does the canonical MU-Floer homotopy type of $T^*\bbC P^n$, $n\in\bbZ$ at least 3 and odd, relate to $\Sigma^\infty_+\calL\bbC P^n$?'', i.e.,  ``What is the statement of the MU-Viterbo isomorphism?'' The purpose of the present article is to provide such an MU-level statement and proof. Moreover, we prove the $\bbS$-level statement for completeness.

\subsection{Main results}

Again, $Q$ is a closed smooth manifold, not necessarily oriented or spin, and $T^*Q$ its cotangent bundle endowed with its standard Liouville structure.

\begin{notation}
We denote by ${\rm Sp}$ the stable $\infty$-category of spectra. If $\frakR$ is an $\mathbb{E}_\infty$-ring spectrum (e.g., $\bbS$ or MU), then ${\rm Mod}_\frakR$ is the stable $\infty$-category of $\frakR$-modules; note, ${\rm Mod}_\bbS={\rm Sp}$. Also, ${\rm BGL}_1(\frakR)$ is the classifying space of $\frakR$-local systems. For a topological space $X$ and an $\frakR$-local system $\Phi:X\to{\rm BGL}_1(\frakR)$ on $X$, we denote by $X^\Phi\in{\rm Mod}_\frakR$ its associated Thom spectrum defined via the colimit
    \begin{equation}
    X^\Phi\equiv\operatorname{colim}\big(X\xrightarrow{\Phi}{\rm BGL}_1(\frakR)\to{\rm Mod}_\frakR\big).
    \end{equation}
\end{notation}

Our first offering is an MU-Viterbo isomorphism.

\begin{thm}\label{thm:mainMU}
Let $\frakF^{{\rm MU}}$ be the canonical ${\rm MU}$-Floer homotopy type of $T^*Q$; there exists an ${\rm MU}$-local system $\Sigma^{-\dim Q}\calV\otimes\epsilon:\calL Q\to{\rm BGL}_1({\rm MU})$ which satisfies
    \begin{equation}\label{eqn:MUViterbo}
    \frakF^{{\rm MU}}\simeq(\calL Q)^{\Sigma^{-\dim Q}\calV\otimes\epsilon}\in{\rm Mod}_{\rm MU}.
    \end{equation}
In particular, we recover the (co)homological statement
    \begin{equation}
    SH^{-*}(T^*Q;\bbZ)\cong H_*(\calL Q;\eta),
    \end{equation}
where $\eta$ is a local system explicitly described in \eqref{eqn:localsystem}.
\end{thm}

\begin{rem}
Our construction of the MU-local system $\Sigma^{-\dim Q}\calV\otimes\epsilon$, in particular $\calV$, is purely through index theory, i.e., Cauchy-Riemann operators. We conjecture a more homotopy-theoretic origin of this MU-local system, cf. Conjecture \ref{conj:main}.
\end{rem}

\begin{rem}
In \cite{BP26}, the notion of a ``Grothendieck-Riemann-Roch (GRR)'' flow category was introduced. These are flow categories with extra structure provided via index theory, and the stable $\infty$-category of such objects recovers the usual notion of complex-oriented flow categories (moreover, the latter recovers MU-modules, cf. \cite{AB24,HO26}). In particular, an index-theoretic model for the endomorphism mapping spectrum of MU was provided in \cite{BP26}. Because of this, it is completely possible to rephrase the construction of our MU-local system $\Sigma^{-\dim Q}\calV\otimes\epsilon$ in terms of this index-theoretic model for MU.
\end{rem}

For completeness, we also provide an $\bbS$-Viterbo isomorphism.

\begin{thm}\label{thm:mainS}
Let $\frakF^{\bbS,\Lambda_\std}$ be the framed Floer homotopy type of $T^*Q$ with respect to the ``standard'' stable $\bbR$-polarization given by the horizontal totally real distribution; this satisfies
    \begin{equation}\label{eqn:SViterbo}
    \frakF^{\bbS,\Lambda_\std}\simeq\Sigma^\infty_+\calL Q\in\spectra.
    \end{equation}
In particular, we recover the (co)homological statement
    \begin{equation}
    SH^{-*}(T^*Q;\eta^{-1})\cong H_*(\calL Q;\bbZ),
    \end{equation}
where $\eta$ is a local system explicitly described in \eqref{eqn:localsystem}.
\end{thm}

\begin{rem}
As already mentioned, it was shown in \cite{BP26} that the canonical MU-Floer homotopy type of $T^*\bbC P^n$, $n\in\bbZ$ at least 3 and odd, cannot possibly be equivalent to $\frakX\otimes_\bbS{\rm MU}$, for any $\frakX\in\spectra$. Observe, in this case, $\bbC P^n$ is spin; however, the MU-local system 
    \begin{equation}
    \Sigma^{-\dim Q}\calV\otimes\epsilon:\calL\bbC P^n\to{\rm BGL}_1({\rm MU}),
    \end{equation}
in this case, \emph{cannot} possibly lift to ${\rm BGL}_1(\bbS)$ by the aforementioned result.
\end{rem}

\subsection{Speculations and symmetry-breaking}
\subsubsection{Symmetry-breaking}\label{part:symmetrybreaking}
We would like to begin this subsection by making explicit the interaction between the various constructions in the present article and the lurking $S^1$-action in the background. The short answer: we explicitly break $S^1$-symmetry in Lemma \ref{lem:symmetrybreaking} in order to account for the non-orientable generality. The long answer is as follows.

Recall, $SH^*(T^*Q;\bbZ)$ admits the structure of a Batalin-Vilkovisky (BV) algebra, and this agrees with the BV algebra structure on $H_{-*}(\calL Q;\eta)$, cf. \cite[Theorem 4.1.1]{Abo14}. In fact, this equivalence refines, at the (co)chain level, to an equivalence of algebras over the framed $E_2$-operad. One subtle point is that, in the case that $Q$ is non-oriented, the equivalence of framed $E_2$-structures should be taken \emph{only} with $\bbZ/2$-gradings. In fact, this can already be seen at the level of BV algebras because, in the non-oriented case, the product and BV operator only satisfy ``twisted graded commutativity'', cf. \cite[Remark 2.1.1]{Abo14}. Meanwhile, if $Q$ is oriented, this issue does not arise and we may upgrade to $\bbZ$-gradings.

Now, the framed $E_2$-structure is related to the $S^1$-action given by loop rotation. Thus, the (co)homological observation of the previous paragraph may have spectral consequences: perhaps we should not expect an MU or $\bbS$-Viterbo isomorphism to respect $S^1$-equivariance in general, since, if it did, we might be on our way to a contradiction with the previous paragraph. On the other hand, we do expect an MUP-Viterbo isomorphism, in general, to respect the framed $E_2$-structures involved, where MUP is the \emph{periodic} complex bordism spectrum. (Recall, the canonical MUP-Floer homotopy type will recover $\bbZ/2$-graded integral symplectic cohomology.)

The reason we bring this point up is to justify our ``$S^1$-symmetry-breaking'' in the present article. In particular, our construction of twisted stable framings on the Floer moduli spaces, and our construction of the MU-local system in Theorem \ref{thm:mainMU}, seem to require us to fix the straight line $s\mapsto(s,1)$ on the cylinder $\bbR\times S^1$ in order to deal with the non-oriented generality, cf. Lemma \ref{lem:symmetrybreaking}.

\begin{rem}
If we were instead working in the oriented case, we could circumvent these issues by utilizing this extra structure. E.g. we could replace the use of Lemma \ref{lem:symmetrybreaking} with the alternative Lemma \ref{lem:nonsymmetrybreaking}.
\end{rem}

\begin{rem}
We should mention Rezchikov's construction of a cyclotomic structure on the $\bbS$-Floer homotopy type, assuming the background Liouville manifold admits an appropriately equivariant stable $\bbR$-polarization, cf. \cite{Rez24}.
\end{rem}

\begin{ques}
This now leads to the following question: what is the statement of the MUP-Viterbo isomorphism when taking into account the framed $E_2$-structures? I.e., what is the MUP-local system on $\calL Q$ which does not break $S^1$-symmetry?
\end{ques}

\subsubsection{Speculations}
Now, some speculations. It is natural to wonder whether there is a purely homotopy-theoretic construction of the MU-local system detailed in Theorem \ref{thm:mainMU}. In fact, we may construct a natural candidate, as follows.

Consider the standard cofiber sequence 
    \begin{equation}
    {\rm BU}\to{\rm BO}\to{\rm B(O/U)}\simeq{\rm O},
    \end{equation}
where the equivalence is real Bott periodicity. Since the composition 
    \begin{equation}
    {\rm BU}\to{\rm BO}\to{\rm BGL}_1(\bbS)\to{\rm BGL}_1({\rm MU})
    \end{equation}
is canonically null (recall, BU is complex-oriented), the universal property of the cofiber yields a map 
    \begin{equation}\label{eqn:canonicalOMU}
    {\rm O}\to{\rm BGL}_1({\rm MU}).
    \end{equation}
Meanwhile, consider a map 
    \begin{equation}
    [TQ]:Q\to{\rm BO}
    \end{equation}
which classifies the stable tangent bundle of $Q$. From here, we may consider the local system
    \begin{equation}\label{eqn:conjecture}
    \widetilde{\calV}:\calL Q\xrightarrow{\calL [TQ]}\calL{\rm BO}\to\Omega{\rm BO}\simeq{\rm O}\to{\rm BGL}_1({\rm MU}),
    \end{equation}
where the second map is the retraction coming from the fact that ${\rm BO}$ is a(n infinite) loop space, the equivalence is real Bott periodicity, and the last map is \eqref{eqn:canonicalOMU}.

\begin{conj}\label{conj:main}
We have that 
    \begin{equation}\label{eqn:conj}
    (\calL Q)^{\Sigma^{-\dim Q}\widetilde{\calV}\otimes\epsilon}\simeq(\calL Q)^{\Sigma^{-\dim Q}\calV\otimes\epsilon}.
    \end{equation}
\end{conj}

\begin{rem}
The map $\calL{\rm BO}\to\Omega{\rm BO}$ breaks $S^1$-symmetry, so we require a proof that, at least in the oriented case, the left hand side of \eqref{eqn:conj} is appropriately $S^1$-equivariant. Of course, if the previous conjecture holds, then the fact that the right hand side of \eqref{eqn:conj} should carry an $S^1$-action yields the desired result; however, one would perhaps want a purely homotopy-theoretic proof.
\end{rem}

\subsection*{Acknowledgments}
The present article essentially began as a short note on the present author's website which only described the framed case; its appearance was prompted by a question asked during his qualifying exam by his advisor: Paul Seidel. The author would like to thank Paul Seidel for initially leading him down this line of thinking; the author would also like to thank the other two qualifying exam committee members: Jeremy Hahn and Tristan Ozuch. The author thanks Thomas Kragh and Noah Porcelli for discussions over the years. This work was partially supported by an NSF Graduate Research Fellowship award.

\section{Flow categories}
We will quickly review the algebraic preliminaries concerning flow categories, following \cite{AB24}. Cf. also \cite{HO26}.

\subsection{Unstructured}
Let $X$ be a smooth manifold with corners; by this, we will always mean a ``$\langle k\rangle$-manifold'' in the sense of \cite{Lau00}. There is a poset associated to $X$, denoted $\calP_X$, such that objects are connected corner strata $F\subset X$ with a morphism $F\to F'$ if and only if $F'\subset\overline{F}$. Moreover, there is functor $\codim:\calP_X\to\bbZ_{\geq0}$ sending each object to its codimension. We say that $X$ is \emph{stratified} by $\calP_X$.\footnote{In the general theory, coarser stratifications than just labelings by corner strata are needed; this leads to the concept of a ``model for manifolds with corners'', cf. \cite[Definition 2.1]{AB24}. In the present article, we will not require general models, so we ignore this point for brevity. Note, $\calP_X$ is an example of a model.} A morphism $X\to X'$ of smooth manifolds with corners is a smooth map which is the inclusion of a union of disjoint boundary faces; this induces a codimension-preserving functor $\calP_X\to\calP_{X'}$. For $p\in\calP_X$, we write $(\partial^pX,\partial^p\calP_X)$ for the stratified smooth manifold with corners consisting of the union of faces in $X$ living under the object $p$ together with the stratifying category given by the undercategory $\partial^p\calP_X\equiv(\calP_X)_{p/}$.

For the following construction, cf. \cite[Section 4]{AB24}. Let $\vec{\calP}\equiv\{\calP_0,\ldots,\calP_n\}$ be any $(n+1)$-tuple of sets, $n\geq0$. For any two 
    \begin{equation}
    x\in\calP_j\;\;{\rm and}\;\;y\in\calP_{j'},\;\;0\leq j\leq j'\leq n,
    \end{equation}
we may construct a category $\vec{\calP}(x,y)$, as follows.
\begin{itemize}
\item The objects are trees, with only bivalent vertices, equipped with the following extra decorations.
    \begin{itemize}
    \item Edges are labeled by elements of $\calP_k$, $j\leq k\leq j'$, such that elements on edges appear in ascending order from left to right with respect to the labels $\{j,\ldots,k,\ldots,j'\}$. Moreover, the incoming resp. outgoing leaf is labeled by $x$ resp. $y$.
    \item Vertices are labeled by a subset of $\{k+1,\ldots,k'-1\}$, where the edge to the left resp. right of the considered vertex is labeled by an element of $\calP_k$ resp. $\calP_{k'}$.
    \end{itemize}
\item There is a morphism $T\to T'$ if the following conditions are satisfied.
    \begin{itemize}
    \item $T$ is obtained from $T'$ by collapsing internal edges.
    \item The labels of the edges of $T$ agree with the labels of the corresponding uncollapsed edges of $T'$.
    \item Given any vertex $v\in T$, the label on $v$ contains the union of the labels of the vertices $v'\in T'$ which are collapsed to $v$ together with any $k$, $j<k<j'$, with the property that (1) $T'$ contains an edge labeled by $\calP_k$ and (2) all edges of $T'$ labeled by $\calP_k$ are collapsed to $v$.
    \end{itemize}
\end{itemize}

For any three $x\in\calP_j$, $y\in\calP_{j'}$, and $z\in\calP_k$, $0\leq j\leq k\leq j'\leq n$, there is a natural functor 
    \begin{equation}
    \vec{\calP}(x,z)\times\vec{\calP}(z,y)\to\vec{\calP}(x,y)
    \end{equation}
given by concatenation. In particular, the collection $\big\{\vec{\calP}(x,y)\big\}_{x,y}$ yields a strict 2-category which we also denote by $\vec{\calP}$. There are natural inclusions of strict 2-categories 
    \begin{equation}
    \partial^j\vec{\calP}\equiv\big(\calP_0,\ldots,\widehat{\calP_j},\ldots,\calP_n\big)\hookrightarrow\vec{\calP},\;\;\forall j;
    \end{equation}
it is straightforward to see that we have the following equality of strict 2-categories: 
    \begin{equation}
    \partial^{j}\partial^{j'}\vec{\calP}=\partial^{j'-1}\partial^{j}\vec{\calP},\;\;0\leq j<j'\leq n.
    \end{equation}

\begin{defin}
An \emph{unstructured flow $n$-simplex} $\bbX\equiv\bbX_{\langle 0\cdots n\rangle}$ on an object $(n+1)$-tuple $\vec{\calP}$ consists of the following data. In the following, we always take 
\begin{equation}
x\in\calP_j,\;\;y\in\calP_{j'},\;\;{\rm and}\;\;z\in\calP_k,\;\;0\leq j\leq k\leq j'\leq n,
\end{equation}
unless otherwise specified.
\begin{enumerate}
\item For any two $x$ and $y$, a stratified compact smooth manifold with corners 
\begin{equation}
\big(\bbX(x,y),\calP_{\bbX(x,y)}=\vec{\calP}(x,y)\big).
\end{equation}
\item For any three $x$, $y$, and $z$, a morphism 
\begin{equation}
\bbX(x,z)\times\bbX(z,y)\hookrightarrow\bbX(x,y)
\end{equation}
which is an embedding of a codimension 1 boundary stratum which lifts 
\begin{equation}
\vec{\calP}(x,z)\times\vec{\calP}(z,y)\to\vec{\calP}(x,y);
\end{equation}
moreover, the natural associativity diagram commutes.
\item For any two $x$ and $y$, the codimension 1 boundary strata of $\bbX(x,y)$ are enumerated by (1) morphisms of the above form and (2) stratified compact smooth submanifolds with corners 
\begin{equation}
\big(\partial^k\bbX(x,y),\calP_{\partial^k\bbX(x,y)}=\partial^k\vec{\calP}(x,y)\big),\;\;j<k<j';
\end{equation}
moreover, the natural associativity diagrams intertwining these commutes.
\end{enumerate}
\end{defin}

\begin{rem}
In the general theory, the previous definition is actually an ``elementary flow $n$-simplex'', cf. \cite[Definition 4.8]{AB24}; a flow $n$-simplex is technically a bit more data, cf. Definition 4.10 in \emph{loc. cit.}. However, since we will only use elementary flow $n$-simplices in the present article, we will elide the distinction. 
\end{rem}

\begin{rem}
A flow 0-simplex is simply referred to as a \emph{flow category}.
\end{rem}

\begin{example}
A simple example of an unstructured flow category is the \emph{unit}, denoted $\unit$; this is an unstructured flow category with a single object and no morphisms. 
\end{example}

We denote by $\flow$ the stable $\infty$-category of unstructured flow categories, cf. \cite[Theorem 1.6]{AB24}.

\subsection{Structured}
The point of a structured flow category (or, more generally, a structured flow $n$-simplex) is to endow the moduli spaces defining the morphism spaces with coherent tangential structure. Two of the most important ones for Floer homotopy theory are as follows.

\subsubsection{(Spherically) complex-oriented}

Let $X$ be a space. Recall, a (complex) virtual bundle on $X$ is a pair 
    \begin{equation}
    E\equiv(E^+,E^-)\equiv E^+-E^-
    \end{equation}
such that $E^\pm\to X$ is a (complex) vector bundle. An equivalence of (complex) virtual bundles, written 
    \begin{equation}
    E\cong F,
    \end{equation}
is an isomorphism of (complex) vector bundles 
    \begin{equation}
    E^+\oplus F^-\cong E^-\oplus F^+.
    \end{equation}

\begin{defin}
A \emph{complex-oriented flow $n$-simplex} consists of the following data. 
\begin{enumerate}
\item A flow $n$-simplex $\bbX$.
\item For any $x$, a virtual vector space $V_x$.
\item For any two $x$ and $y$, a vector bundle $W(x,y)\to\bbX(x,y)$ satisfying 
    \begin{equation}
    W(x,y)\cong W(x,z)\oplus W(z,y)
    \end{equation}
over $\bbX(x,z)\times\bbX(z,y)$, such that the natural associativity diagram commutes.
\item For any two $x$ and $y$, a complex vector bundle $I(x,y)\to\bbX(x,y)$ satisfying 
    \begin{equation}
    I(x,y)\cong I(x,z)\oplus I(z,y)
    \end{equation}
over $\bbX(x,z)\times\bbX(z,y)$, such that the natural associativity diagram commutes.
\item For any two $x$ and $y$, an equivalence of virtual bundles
    \begin{equation}
    T\bbX(x,y)+\underline{V}_y+W(x,y)\cong I(x,y)+W(x,y)+\underline{\bbR}^{\abs{\{j+1,\ldots,j'\}}}+\underline{V}_x,
    \end{equation}
where $x\in\calP_j$ and $y\in\calP_{j'}$, such that the natural associativity diagram commutes.
\end{enumerate}
\end{defin}

We denote by $\flow^{\rm cx}$ the stable $\infty$-category of complex-oriented flow categories, cf. \cite[Theorem 1.6]{AB24}; in fact, 
    \begin{equation}
    \flow^{\rm cx}\cong{\rm Mod}_{\rm MU}
    \end{equation}
as stable $\infty$-categories, cf. \cite{HO26}. 

In the present article, we will require the following more general notion. We denote by $\frakR^{\rm scx}$ the spectrum defined via
    \begin{equation}
    \frakR^{\rm scx}\equiv{\rm colim}\Big(\rm{fib}\big({\rm BO}\to{\rm BGL}_1(\bbS)\to{\rm BGL}_1({\rm MU})\big)\to{\rm BO}\to{\rm BGL}_1(\bbS)\to{\rm Sp}\Big).
    \end{equation}
Observe, $\frakR^{\rm scx}$ is the Thom spectrum classifying spherical complex bordism, i.e., bordism of ${\rm MU}$-oriented smooth manifolds, and we have a natural map $i_{\rm MU}:{\rm MU}\to\frakR^{\rm scx}$. A flow categorical model for ${\rm Mod}_{\frakR^{\rm scx}}$ is as follows. 

\begin{defin}
A \emph{spherically complex-oriented flow $n$-simplex} consists of the following data. 
\begin{enumerate}
\item A flow $n$-simplex $\bbX$.
\item For any $x$, a virtual vector space $V_x$.
\item For any two $x$ and $y$, a vector bundle $W(x,y)\to\bbX(x,y)$ satisfying 
    \begin{equation}
    W(x,y)\cong W(x,z)\oplus W(z,y)
    \end{equation}
over $\bbX(x,z)\times\bbX(z,y)$, such that the natural associativity diagram commutes.
\item For any two $x$ and $y$, a virtual bundle $I(x,y)\to\bbX(x,y)$ with a fixed trivialization after taking its associated stable spherical fibration and smashing with ${\rm MU}$, 
    \begin{equation}
    I(x,y)^\bbS\otimes_\bbS{\rm MU},
    \end{equation} 
satisfying 
    \begin{equation}
    I(x,y)\cong I(x,z)\oplus I(z,y)
    \end{equation}
over $\bbX(x,z)\times\bbX(z,y)$, such that the natural associativity diagram commutes.
\item For any two $x$ and $y$, an equivalence of virtual bundles
    \begin{equation}
    T\bbX(x,y)+\underline{V}_y+W(x,y)\cong I(x,y)+W(x,y)+\underline{\bbR}^{\abs{\{j+1,\ldots,j'\}}}+\underline{V}_x,
    \end{equation}
where $x\in\calP_j$ and $y\in\calP_{j'}$, such that the natural associativity diagram commutes.
\end{enumerate}
\end{defin}

We denote by $\flow^{\rm scx}$ the stable $\infty$-category of spherically complex-oriented flow categories, cf. \cite[Theorem 1.6]{AB24}. And, as before, 
    \begin{equation}
    \flow^{\rm scx}\cong{\rm Mod}_{\frakR^{\rm scx}}
    \end{equation}
as stable $\infty$-categories, cf. \cite{HO26}. 

We will require the following in the sequel. 

\begin{lem}\label{lem:conservative}
The base-change functor 
    \begin{equation}
    -\otimes_{\rm MU}\frakR^{\rm scx}:{\rm Mod}_{\rm MU}\to{\rm Mod}_{\frakR^{\rm scx}}
    \end{equation}
is conservative.
\end{lem}

\begin{proof}
This follows because there exists a natural retraction splitting $i_{\rm MU}$, cf. \cite[Lemma 3.15]{AB19}.
\end{proof}

\subsubsection{(Spherically) framed}

\begin{defin}
A \emph{framed flow $n$-simplex} consists of the following data. 
\begin{enumerate}
\item A flow $n$-simplex $\bbX$.
\item For any $x$, a virtual vector space $V_x$.
\item For any two $x$ and $y$, a vector bundle $W(x,y)\to\bbX(x,y)$ satisfying 
    \begin{equation}
    W(x,y)\cong W(x,z)\oplus W(z,y)
    \end{equation}
over $\bbX(x,z)\times\bbX(z,y)$, such that the natural associativity diagram commutes.
\item For any two $x$ and $y$, a trivial vector bundle $I(x,y)\to\bbX(x,y)$ satisfying 
    \begin{equation}
    I(x,y)\cong I(x,z)\oplus I(z,y)
    \end{equation}
over $\bbX(x,z)\times\bbX(z,y)$, such that the natural associativity diagram commutes.
\item For any two $x$ and $y$, an equivalence of virtual bundles
    \begin{equation}
    T\bbX(x,y)+\underline{V}_y+W(x,y)\cong I(x,y)+W(x,y)+\underline{\bbR}^{\abs{\{j+1,\ldots,j'\}}}+\underline{V}_x,
    \end{equation}
where $x\in\calP_j$ and $y\in\calP_{j'}$, such that the natural associativity diagram commutes.
\end{enumerate}
\end{defin}

We denote by $\flow^{\rm fr}$ the stable $\infty$-category of framed flow categories, cf. \cite[Theorem 1.6]{AB24}; in fact, 
    \begin{equation}
    \flow^{\rm fr}\cong{\rm Sp}
    \end{equation}
as stable $\infty$-categories, cf. Proposition 1.10 in \emph{loc. cit.} 

Analogously to spherically complex-oriented flow categories, we may define the stable $\infty$-category of spherically framed flow categories, denoted $\flow^{\rm sfr}$; this satisfies 
    \begin{equation}
    \flow^{\rm sfr}\cong{\rm Mod}_{\frakR^{\rm sfr}},
    \end{equation}
where $\frakR^{\rm sfr}$ is the Thom spectrum classifying spherical framed bordism:
    \begin{equation}
    \frakR^{\rm sfr}\equiv{\rm colim}\Big(\rm{fib}\big({\rm BO}\to{\rm BGL}_1(\bbS)\big)\to{\rm BO}\to{\rm BGL}_1(\bbS)\to{\rm Sp}\Big).
    \end{equation}

\section{Free loop spaces}
We begin with string topology. This section largely follows, and lifts to spectra, \cite[Chapter 3]{Abo14}.

\subsection{Finite-dimensional approximations}
Let $Q$ be a closed smooth manifold. We fix a Riemannian metric $g$ on $Q$ of sufficiently large injectivity radius; moreover, we denote by $d(\cdot,\cdot)$ the associated distance function. We will probe the topology of $\calL Q$ using finite-dimensional approximations. Essentially, this allows us to use Morse theory to study the topology of $\calL Q$.

For each $r\in\bbZ_{>0}$, we define 
    \begin{equation}
    \calL^r Q\equiv\bigcap_{i=1}^r\rho^{-1}_i\big((-\infty,\delta^r_i]\big)\subset Q^r, \;\;\delta^r_i\in\bbR_{\geq0},
    \end{equation}
where 
    \begin{align}
    \rho_i:Q^r&\to\bbR \\ 
    (q_0,\ldots,q_{r-1})&\mapsto d(q_i,q_{i+1}) \nonumber
    \end{align}
with $i$ a cyclic index.

\begin{lem}[Lemma 3.2.6 in \cite{Abo14}]
For a generic choice of $\{\delta^r_i\}_i$ sufficiently small, $\calL^rQ$ is a smooth manifold with corners.
\end{lem}

Moreover, by choosing $\{\delta^r_i\}_{r,i}$ to satisfy $\delta^r_i\leq\delta^{r+1}_{i+1}$, $1\leq i\leq r$, we have natural inclusions 
    \begin{equation}
    i_{r,r+1}:\calL^rQ\hookrightarrow\calL^{r+1}Q
    \end{equation}
given by repeating the first point. Meanwhile, for each $r$, there is a natural map 
    \begin{equation}
    {\rm geo}:\calL^rQ\to\calL Q
    \end{equation}
defined by mapping a collection of points to the piecewise geodesic connecting them (parameterized at unit speed); this is compatible with the inclusions. Finally, we assume $\lim_{r\to+\infty}\sum_{i=1}^r\delta^r_i=+\infty$.

\begin{lem}[Proposition 3.2.4 in \cite{Abo14}]
We have that
    \begin{equation}\label{eqn:fda}
    \varinjlim_r\calL^rQ\simeq\calL Q.
    \end{equation}
\end{lem}

\subsection{Morse homotopy type}
In this subsection, we aim to recast \eqref{eqn:fda} in terms of Morse theory. We will do this via the Morse homotopy type.

\subsubsection{Flow categories and the finite-dimensional approximations}
We denote by $g^r$ the $r$-fold product metric on $Q^r$. Let $f^r\in C^\infty(\calL^rQ)$ be a Morse function whose gradient $\grad f^r$ points strictly outwards along the boundary $\partial\calL^r Q$. For any two critical points $a,b\in\crit(f^r)$, we denote by $\widetilde{\scrM}^r(a,b)$ the moduli space of Morse trajectories connecting $a$ to $b$, i.e., smooth maps $\gamma:\bbR_s\to\calL^rQ$ satisfying 
    \begin{equation}
    \begin{cases}
    \partial_s\gamma=-\grad f^r(\gamma) \\
    \lim_{s\to-\infty}\gamma(s)=a \\
    \lim_{s\to+\infty}\gamma(s)=b.
    \end{cases}
    \end{equation}
Observe, this moduli space may be viewed as the zero set of a section $\Xi_r$ of an appropriate Banach bundle over the space of all $W^{1,2}$-paths connecting $a$ to $b$; in particular, it is a smooth manifold of dimension $I(a)-I(b)$, with $I(\cdot)$ the Morse index, when the corresponding linearized operator $D(\Xi_r)_\gamma$ is a surjective Fredholm operator for each $\gamma\in\widetilde{\scrM}^r(a,b)$ --- this happens for generic $(f^r,g^r)$. We call such data \emph{Morse-Smale}. Moreover, $T\widetilde{\scrM}^r(a,b)$ is classified by the index bundle $\ind D\Xi_r$ of the family of Fredholm operators 
    \begin{equation}
    D\Xi_r\equiv\big\{D(\Xi_r)_\gamma\big\}_\gamma.
    \end{equation}

When $a\neq b$, there is a free, proper $\bbR$-action on $\widetilde{\scrM}^r(a,b)$ given by time-shift in the $s$-coordinate, and the $\bbR$-quotient $\scrM^r(a,b)$ admits a Gromov-compactification $\bbM^r(a,b)$ given by broken Morse trajectories. In fact, $\bbM^r(a,b)$ is a compact smooth manifold with corners, stratified by breakings at critical points, whose codimension 1 boundary strata are enumerated by gluing maps of the form 
    \begin{equation}
    \partial_c\equiv\bbM^r(a,c)\times\bbM^r(c,b)\hookrightarrow\bbM^r(a,b),
    \end{equation}
cf. for instance \cite{Weh12} (this is also a simpler version of \cite[Section 6]{Lar21}). 

We have the following two basic relations: 
\begin{enumerate}
\item a short exact sequence 
    \begin{equation}
    0\to\underline{\bbR}\to T\bbM^r(a,b)\vert_{\partial_c}\to T\bbM^r(a,c)\oplus T\bbM^r(c,b)\to0
    \end{equation}
given by taking a collar neighborhood of $\partial_c$, 
\item and a short exact sequence 
    \begin{equation}
    0\to\underline{\bbR}\to\ind D\Xi^{ab}_r\to T\bbM^r(a,b)\vert_{\operatorname{int}\bbM^r(a,b)}\to0
    \end{equation}
given by the translational direction.
\end{enumerate}
Moreover, there are straightforward relations when passing to higher codimension boundary strata. The proof of the following result is a simpler version of \cite[Section 7]{Lar21} (cf. also \cite[Section 8]{PS24} and \cite[Section 6]{PS25c}).

\begin{prop}
There is an extension of $\ind D\Xi^{ab}_r$ to $\bbM^r(a,b)$ whose restriction to the interior of a codimension 1 boundary stratum is of the form 
    \begin{equation}
    \ind D\Xi^{ab}_r\vert_{\partial_c}=\ind D\Xi^{ac}_r\oplus\ind D\Xi^{cb}_r.
    \end{equation}
The natural associativity diagram,
    \begin{equation}
    \begin{tikzcd}
    & & \underline{\bbR}\arrow[d] \\
    \underline{\bbR}\arrow[d,"\Delta"]\arrow[r] & \ind D\Xi^{ab}_r\vert_{\partial_c}\arrow[r]\arrow[d,equals] & T\bbM^r(a,b)\vert_{\partial_c}\arrow[d] \\
    \underline{\bbR}^2\arrow[r] & \ind D\Xi^{ac}_r\oplus\ind D\Xi^{cb}_r\arrow[r] & T\bbM^r(a,c)\oplus T\bbM^r(c,b),
    \end{tikzcd}
    \end{equation}
commutes. Moreover, the natural associativity diagram associated to higher codimension boundary strata commute.
\end{prop}

Let $\bbM^r$ be the unstructured flow category with objects $\crit(f^r)$ and morphism spaces $\bbM^r(a,b)$; the composition map is gluing of Morse trajectories. In fact, there is a standard way to construct stable framings on the Morse moduli spaces, 
    \begin{equation}\label{eqn:morseframing}
    T\bbM^r(a,b)+\underline{\bbR}^{I(b)}+\underline{\bbR}\cong\underline{\bbR}^{I(a)},
    \end{equation}
which are compatible with gluing.

\begin{lem}[Proposition 4.9 in \cite{Bla24}]
$\bbM^r$ lifts to a framed flow category $\bbM^{\bbS,r}$ satisfying
    \begin{equation}
    \frakM^{\bbS,r}\equiv\flow^{\fr}(\unit,\bbM^{\bbS,r})\simeq\Sigma^\infty_+\calL^rQ.
    \end{equation}
\end{lem}

Now, if $\Phi:\calL^rQ\to{\rm BGL}_1(\bbS)$ is any stable spherical fibration, then we may twist both sides of \eqref{eqn:morseframing},
    \begin{equation}
    T\bbM^r(a,b)+\underline{\bbR}^{I(b)}+\underline{\Phi}\vert_b+\underline{\bbR}\cong I(a,b)+\underline{\bbR}^{I(a)}+\underline{\Phi}\vert_a,
    \end{equation}
where: we abuse notation to think of $\underline{\Phi}\vert_a$ and $\underline{\Phi}\vert_b$ as virtual vector spaces of virtual rank 0, and $I(a,b)$ is a virtual bundle of virtual rank 0 which is trivial as a stable spherical fibration; to obtain a new spherically framed flow category $\bbM^{{\rm sfr},r,\Phi}$.

\begin{lem}
For $\Phi:\calL^rQ\to{\rm BGL}_1(\bbS)$, we have that $\bbM^{{\rm sfr},r,\Phi}$ satisfies
    \begin{equation}
    \frakM^{{\rm sfr},r,\Phi}\equiv\flow^{\rm sfr}(\unit,\bbM^{{\rm sfr},r,\Phi})\simeq(\calL^rQ)^{\Phi}\otimes_\bbS\frakR^{\rm sfr}.
    \end{equation}
\end{lem}

In fact, if $\Phi:\calL^rQ\to\bbZ\times{\rm BO}$ is simply a virtual bundle, then we may twist both sides of \eqref{eqn:morseframing},
    \begin{equation}
    T\bbM^r(a,b)+\underline{\bbR}^{I(b)}+\underline{\Phi}\vert_b+\underline{\bbR}\cong\underline{\bbR}^{I(a)}+\underline{\Phi}\vert_a,
    \end{equation}
to obtain a new framed flow category $\bbM^{\bbS,r,\Phi}$.

\begin{lem}
For $\Phi:\calL^rQ\to\bbZ\times{\rm BO}$, we have that $\bbM^{\bbS,r,\Phi}$ satisfies
    \begin{equation}
    \frakM^{\bbS,r,\Phi}\equiv\flow^{\fr}(\unit,\bbM^{\bbS,r,\Phi})\simeq(\calL^rQ)^{\Phi}.
    \end{equation}
\end{lem}

Of course, we may forget the framed flow category $\bbM^{\bbS,r}$ down to a complex-oriented flow category $\bbM^{\rm MU,r}$ satisfying 
    \begin{equation}
    \frakM^{{\rm MU},r}\equiv\flow^{{\rm cx}}(\unit,\bbM^{{\rm MU},r})\simeq\Sigma^\infty_+\calL^rQ\otimes_\bbS{\rm MU}.
    \end{equation}
    
Moreover, if $\Phi:\calL^rQ\to{\rm BGL}_1({\rm MU})$ is any MU-local system, we may twist both sides of \eqref{eqn:morseframing},
    \begin{equation}
    T\bbM^r(a,b)+\underline{\bbR}^{I(b)}+\underline{\Phi}\vert_b+\underline{\bbR}\cong I(a,b)+\underline{\bbR}^{I(a)}+\underline{\Phi}\vert_a,
    \end{equation}
where $I(a,b)$ is a virtual bundle of virtual rank 0 which is trivial as an MU-local system, to obtain a new spherically complex-oriented flow category $\bbM^{{\rm scx},r,\Phi}$.

\begin{lem}
For $\Phi:\calL^rQ\to{\rm BGL}_1({\rm MU})$, we have that $\bbM^{{\rm scx},r,\Phi}$ satisfies
    \begin{equation}
    \frakM^{{\rm scx},r,\Phi}\equiv\flow^{\rm scx}(\unit,\bbM^{{\rm scx},r,\Phi})\simeq(\calL^rQ)^{\Phi}\otimes_{\rm MU}\frakR^{\rm scx}.
    \end{equation}
\end{lem}

\subsubsection{Flow bimodules and the directed system}
Given $a\in\crit(f^r)$, we denote by 
    \begin{equation}
    W^u(a;f^r)\;\;{\rm resp.}\;\;W^s(a;f^r)
    \end{equation}
the unstable resp. stable manifold of $a$. We may consider the intersection 
    \begin{equation}\label{eqn:aux1}
    i_{r,r+1}\big(W^u(a;f^r)\big)\cap W^s(b;f^{r+1})
    \end{equation}
which, for generic data, is a smooth manifold of dimension $I(a)-I(b)$. In fact, \eqref{eqn:aux1} admits a Gromov-compactification $\bbI_{r,r+1}(a,b)$, given by allowing breakings at critical points, which, for generic data, is a compact smooth manifold with corners, stratified by breakings at critical points, whose codimension 1 boundary strata are enumerated by gluing maps of the form
    \begin{align}
    \bbM^r(a,a')\times\bbI_{r,r+1}(a',b)&\hookrightarrow\bbI_{r,r+1}(a,b), \\
    \bbI_{r,r+1}(a,b')\times\bbM^{r+1}(b',b)&\hookrightarrow\bbI_{r,r+1}(a,b).
    \end{align}
    
By looking at the various short exact sequences of the form
    \begin{equation}
    0\to T\Big(i_{r,r+1}\big(W^u(a;f^r)\big)\cap W^s(b;f^{r+1})\Big)\to TW^u(a;f^r)\to TW^u(b;f^{r+1})\to0,
    \end{equation}
we have the standard stable framings 
    \begin{equation}\label{eqn:aux2*}
    T\bbI_{r,r+1}(a,b)+\underline{\bbR}^{I(b)}\cong\underline{\bbR}^{I(a)}
    \end{equation}
which are compatible with gluing. In particular, we have constructed a framed flow bimodule 
    \begin{equation}
    \bbI_{r,r+1}:\bbM^{\bbS,r}\to\bbM^{\bbS,r+1},
    \end{equation}
hence a map 
    \begin{equation}
    \bbI_{r,r+1}:\frakM^{\bbS,r}\to\frakM^{\bbS,r+1}
    \end{equation}
which agrees, on integral homology, with the map induced on integral homology by $i_{r,r+1}$. Since every spectrum in sight is bounded below, we may use a spectral Whitehead theorem for $H\bbZ$-module spectra together with \cite[Proposition 3.3.13 \& (3.3.47)]{Abo14} to obtain the following result.

\begin{prop}
We have the following homotopy equivalences:
    \begin{equation}
    \varinjlim_r\frakM^{\bbS,r}\simeq\varinjlim_r\Sigma^\infty_+\calL^rQ\simeq\Sigma^\infty_+\calL Q.
    \end{equation}
\end{prop}

By twisting both sides of \eqref{eqn:aux2*} by a stable spherical fibration (or virtual bundle, or MU-local system) on $\calL Q$, we obtain the following.

\begin{prop}
If $\Phi:\calL Q\to{\rm BGL}_1(\bbS)$, then
    \begin{equation}
    \varinjlim_r\frakM^{{\rm sfr},r,\Phi}\simeq\varinjlim_r\Big((\calL^rQ)^\Phi\otimes_\bbS\frakR^{\rm sfr}\big)\simeq(\calL Q)^\Phi\otimes_\bbS\frakR^{\rm sfr}.
    \end{equation}
\end{prop}

\begin{prop}
If $\Phi:\calL Q\to\bbZ\times{\rm BO}$, then
    \begin{equation}
    \varinjlim_r\frakM^{\bbS,r,\Phi}\simeq\varinjlim_r(\calL^rQ)^\Phi\simeq(\calL Q)^\Phi.
    \end{equation}
\end{prop}

\begin{prop}
If $\Phi:\calL Q\to{\rm BGL}_1({\rm MU})$, then
    \begin{equation}
    \varinjlim_r\frakM^{{\rm scx},r,\Phi}\simeq\varinjlim_r\big((\calL^rQ)^\Phi\otimes_{\rm MU}\frakR^{\rm scx}\big)\simeq(\calL Q)^\Phi\otimes_{\rm MU}\frakR^{\rm scx}.
    \end{equation}
\end{prop}

\section{Cotangent bundles}
We continue with Floer theory; this section follows, and lifts to spectra, \cite[Chapter 2]{Abo14}.

\subsection{Canonical complex-oriented Floer homotopy type}
Recall, we endow $T^*Q$ with its standard Liouville structure $\omega=d\theta$. We denote by $\rho(\cdot,\cdot)$ the radial coordinate on $T^*Q$ defined via taking the norm of a covector; in particular, 
    \begin{equation}
    T^*Q=D^*Q\cup_{S^*Q}\big(S^*Q\times[1,+\infty)\big),
    \end{equation}
where $D^*Q$ resp. $S^*Q$ is the unit co-disk resp. co-sphere bundle of $Q$.

Let $H\in C^\infty(S^1\times T^*Q)$ be a Hamiltonian with associated Hamiltonian vector field $X_H$ defined via
    \begin{equation}
    \omega(X_H,\cdot)=dH(\cdot);
    \end{equation}
we denote by $\phi^t_H$ the flow of $X_H$. We will take $H$ to be linear at infinity of slope $\tau\in\bbR_{\geq0}$, i.e., $H=\tau\cdot\rho$ near infinity. Moreover, we choose $\tau$ such that it is not the length of any geodesic on $Q$. Recall, a 1-periodic Hamiltonian orbit of $H$ is a smooth map $x:S^1\to T^*Q$ satisfying 
    \begin{equation}
    \partial_tx=X_H(x);
    \end{equation}
we denote by $\chi(H)$ the set of such objects. Meanwhile, let $J$ be an $S^1$-family of $\omega$-compatible almost complex structures on $T^*Q$ which is convex near infinity, where the latter means $d\rho\circ J=-e^h\theta$ in a neighborhood of $S^*Q$ for some locally defined smooth function $h$. We call pairs $(H,J)$ satisfying the conditions above \emph{admissible}.

Let $(H^\ell,J^\ell)$ be admissible Floer data. For any two $x,y\in\chi(H^\ell)$, we denote by $\widetilde{\scrF}^\ell(y,x)$ the moduli space of Floer trajectories connecting $x$ to $y$, i.e., smooth maps $u:\bbR_s\times S^1_t\to T^*Q$ satisfying
    \begin{equation}
    \begin{cases}
    \partial_su+J\big(\partial_t-X_H(u)\big)=0 \\
    \lim_{s\to-\infty}u(s,t)=x(t) \\
    \lim_{s\to+\infty}u(s,t)=y(t).
    \end{cases}
    \end{equation}
Observe, this moduli space may be viewed as the zero set of a section $\overline{\partial}_{H^\ell,J^\ell}$ of an appropriate Banach bundle over the space of all $W^{k,p}$-cylinders, $kp>2$, connecting $x$ to $y$; in particular, it is a smooth manifold of dimension $\deg(x)-\deg(y)$ when the corresponding linearized operator $D(\overline{\partial}_{H^\ell,J^\ell})_u$ is a surjective Fredholm operator for each $u\in\widetilde{\scrF}^\ell(y,x)$ -- this happens for generic $(H^\ell,J^\ell)$. We call such data \emph{regular}. Here, $\deg(x)$ is defined in terms of the Conley-Zehnder index ${\rm CZ}(x)$: 
    \begin{equation}
    \deg(x)\equiv\dim Q-{\rm CZ}(x)+w(x),
    \end{equation}
where $w(x)$ is 0 resp. $-1$ if $x^*TQ$ is orientable resp. non-orientable, cf. \cite[Section 1.4]{Abo14}. Moreover, $T\widetilde{\scrF}^\ell(y,x)$ is classified by the index bundle $\ind D\overline{\partial}_{H^\ell,J^\ell}$ of the family of Fredholm operators 
    \begin{equation}
    D\overline{\partial}_{H^\ell,J^\ell}\equiv\big\{D(\overline{\partial}_{H^\ell,J^\ell})_u\big\}_u.
    \end{equation}

When $x\neq y$, there is a free, proper $\bbR$-action on $\widetilde{\scrF}^\ell(y,x)$ given by time-shift in the $s$-coordinate, and the $\bbR$-quotient $\scrF^\ell(y,x)$ admits a Gromov-compactification $\bbF^\ell(y,x)$ given by broken Floer trajectories. In fact, $\bbF^\ell(y,x)$ is a compact smooth manifold with corners, stratified by breakings at Hamiltonian orbits, whose codimension 1 boundary strata are enumerated by gluing maps of the form 
    \begin{equation}
    \partial_z\equiv\bbF^\ell(y,z)\times\bbF^\ell(z,x)\hookrightarrow\bbF^\ell(y,x),
    \end{equation}
cf. \cite[Section 6]{Lar21}. 

We have the following two basic relations: 
\begin{enumerate}
\item a short exact sequence 
    \begin{equation}
    0\to\underline{\bbR}\to T\bbF^\ell(y,x)\vert_{\partial_z}\to T\bbF^\ell(y,z)\oplus T\bbF^\ell(z,x)\to0
    \end{equation}
given by taking a collar neighborhood of $\partial_z$, 
\item and a short exact sequence 
    \begin{equation}
    0\to\underline{\bbR}\to\ind D\overline{\partial}^{yx}_{H^\ell,J^\ell}\to T\bbF^r(y,x)\vert_{\operatorname{int}\bbF^\ell(y,x)}\to0
    \end{equation}
given by the translational direction.
\end{enumerate}
Moreover, there are straightforward relations when passing to higher codimension boundary strata. The proof of the following result is contained in \cite[Section 7]{Lar21} (cf. also \cite[Section 8]{PS24} and \cite[Section 6]{PS25c}).

\begin{prop}\label{prop:indexfloer}
There is an extension of $\ind D\overline{\partial}^{yx}_{H^\ell,J^\ell}$ to $\bbF^\ell(y,x)$ whose restriction to the interior of a codimension 1 boundary stratum is of the form 
    \begin{equation}
    \ind D\overline{\partial}^{yx}_{H^\ell,J^\ell}\vert_{\partial_z}=\ind D\overline{\partial}^{yz}_{H^\ell,J^\ell}\oplus\ind D\overline{\partial}^{zx}_{H^\ell,J^\ell}.
    \end{equation}
The natural associativity diagram,
    \begin{equation}
    \begin{tikzcd}
    & & \underline{\bbR}\arrow[d] \\
    \underline{\bbR}\arrow[d,"\Delta"]\arrow[r] & \ind D\overline{\partial}^{yx}_{H^\ell,J^\ell}\vert_{\partial_z}\arrow[r]\arrow[d,equals] & T\bbF^\ell(y,x)\vert_{\partial_z}\arrow[d] \\
    \underline{\bbR}^2\arrow[r] & \ind D\overline{\partial}^{yz}_{H^\ell,J^\ell}\oplus\ind D\overline{\partial}^{zx}_{H^\ell,J^\ell}\arrow[r] & T\bbF^\ell(y,z)\oplus T\bbF^\ell(z,x),
    \end{tikzcd}
    \end{equation}
commutes. Moreover, the natural associativity diagram associated to higher codimension boundary strata commute.
\end{prop}

We may assume that we have a Hermitian connection $\nabla$ on $T(T^*Q)$ such that, given any $x\in\chi(H^\ell)$, $\nabla$ is flat in a neighborhood of the image of $x$. Let $\bbF^\ell$ be the unstructured flow category with objects $\chi(H^\ell)$ and morphisms spaces $\bbF^\ell(y,x)$; the composition map is gluing of Floer trajectories. There is a canonical way to construct stable complex structures on the Floer moduli spaces, 
    \begin{equation}\label{eqn:canonicalMUFloer}
    T\bbF^\ell(y,x)+\underline{\bbR}^{-\deg(x)}+\underline{\bbR}\cong I(y,x)+\underline{\bbR}^{-\deg(y)},
    \end{equation}
as follows, cf. \cite[Proposition 3.35 \& 3.40]{BP26}. 

Given $u\in\operatorname{int}\bbF^\ell(y,x)$, we consider a Fredholm operator 
    \begin{equation}\label{eqn:analogouscomplex}
    \overline{D(\overline{\partial}_{H^\ell,J^\ell})_u}:W^{k,p}(\bbR\times S^1;\bbC^{\dim Q})\to W^{k-1,p}(\bbR\times S^1;\bbC^{\dim Q})
    \end{equation}
whose asymptotic condition at $-\infty$ resp. $+\infty$ is the conjugate of the asymptotic condition of $D(\overline{\partial}_{H^\ell,J^\ell})_u$ at $-\infty$ resp. $+\infty$; this operator canonically satisfies
    \begin{equation}
    \ind\overline{D(\overline{\partial}_{H^\ell,J^\ell})_u}\cong\bbR^{\deg(y)-\deg(x)}.
    \end{equation}
In particular, we set 
    \begin{equation}
    I(y,x)\equiv\ind\Big(D\overline{\partial}_{H^\ell,J^\ell}\oplus\overline{D\overline{\partial}_{H^\ell,J^\ell}}\Big),
    \end{equation}
and we see 
    \begin{equation}
    T\scrF^\ell(y,x)+\underline{\bbR}^{-\deg(x)}+\underline{\bbR}\cong I(y,x)+\underline{\bbR}^{-\deg(y)}.
    \end{equation}
The fact that this complex orientation (1) extends over the Gromov-compactification and (2) has the correct restriction to a codimension 1 boundary stratum both follow analogously to the proof of Corollary \ref{cor:extension}.

\begin{lem}
$\bbF^\ell$ lifts to a complex-oriented flow category $\bbF^{{\rm MU},\ell}$ satisfying 
    \begin{equation}
    H_*(\frakF^{{\rm MU},\ell};\bbZ)\cong HF^{-*}(T^*Q;H^\ell,J^\ell;\bbZ),\;\;\frakF^{{\rm MU},\ell}\equiv{\rm Flow}^{\rm cx}(\unit,\bbF^{{\rm MU},\ell}).
    \end{equation}
\end{lem}

Let $(H^\ell,J^\ell)$ and $(H^{\ell+1},J^{\ell+1})$ be regular Floer data such that the slopes satisfy $\tau_\ell\leq\tau_{\ell+1}$. As before, we may assume that $\nabla$ is flat in a neighborhood of the image of $x$, for any $\ell$ and $x\in\chi(H^\ell)$. We choose monotonic continuation data $(H_s,J_s)$ connecting $(H^{\ell+1},J^{\ell+1})$ to $(H^\ell,J^\ell)$, i.e., a non-degenerate $\bbR$-dependent family of Hamiltonians, linear at infinity of decreasing slope in $s$, connecting $H^{\ell+1}$ to $H^\ell$ and an $\bbR$-dependent family of ($S^1$-families of) $\omega$-compatible almost complex structures $J_s$, convex at infinity, connecting $J^{\ell+1}$ to $J^\ell$. For any $x\in\chi(H^\ell)$ and $y\in\chi(H^{\ell+1})$, we denote by $\widetilde{\frakc}_{\ell,\ell+1}(x,y)$ the moduli space of Floer continuation cylinders connecting $y$ to $x$, i.e., smooth maps $u:\bbR\times S^1\to T^*Q$ satisfying 
    \begin{equation}
    \begin{cases}
    \partial_su+J\big(\partial_t-X_{H_s}(u)\big)=0 \\
    \lim_{s\to-\infty}u(s,t)=y(t) \\
    \lim_{s\to+\infty}u(s,t)=x(t).
    \end{cases}
    \end{equation}
    
Analogously to before, $\widetilde{\frakc}_{\ell,\ell+1}(x,y)$ admits a Gromov-compactification $\frakc_{\ell,\ell+1}(x,y)$ given by broken Floer continuation cylinders; this is a compact smooth manifold with corners, stratified by breakings at Hamiltonian orbits, whose codimension 1 boundary strata are enumerated by gluing maps of the form 
    \begin{align}
    \bbF^\ell(x,x')\times\frakc_{\ell,\ell+1}(x',y)&\hookrightarrow\frakc_{\ell,\ell+1}(x,y), \\
    \frakc_{\ell,\ell+1}(x,y')\times\bbF^{\ell+1}(y',y)&\hookrightarrow\frakc_{\ell,\ell+1}(x,y).
    \end{align}
Cf. \cite[Section 6]{Lar21} for the transversality statements.

Moreover, analogously to before, there is a canonical way to construct stable complex structures: 
    \begin{equation}\label{eqn:canonicalMUFloercontiuation}
    T\frakc_{\ell,\ell+1}(x,y)+\underline{\bbR}^{-\deg(y)}\cong I(y,x)+\underline{\bbR}^{-\deg(x)}.
    \end{equation}
In particular, we have constructed a complex-oriented flow bimodule 
    \begin{equation}
    \frakc^{\rm MU}_{\ell,\ell+1}:\bbF^{{\rm MU},\ell}\to\bbF^{{\rm MU},\ell+1},
    \end{equation}
hence a map
    \begin{equation}
    \frakc^{\rm MU}_{\ell,\ell+1}:\frakF^{{\rm MU},\ell}\to\frakF^{{\rm MU},\ell+1}
    \end{equation}
which agrees, on integral homology, with the usual Floer continuation map. 

\begin{prop}
We have that $\frakF^{\rm MU}\equiv\varinjlim_\ell\frakF^{{\rm MU},\ell}$ satisfies
    \begin{equation}
    H_*(\frakF^{\rm MU};\bbZ)\cong SH^{-*}(T^*Q;\bbZ),
    \end{equation}
where the colimit is over a family of increasing slope.
\end{prop}

\subsection{Framed Floer homotopy type}
Recall, we denote by $\Lambda_\std$ the ``standard'' stable $\bbR$-polarization on $T^*Q$ induced by the Levi-Civita connection on $Q$ with respect to some Riemannian metric:
    \begin{equation}
    T(T^*Q)\cong\Lambda_\std\otimes_\bbR\underline{\bbC},\;\;\Lambda_\std\equiv\pi^*TQ,\;\;\pi:T^*Q\to Q.
    \end{equation}
We may construct standard twisted stable framings on the Floer moduli spaces, 
    \begin{equation}
    T\bbF^\ell(y,x)+\underline{\ind D}^\calA_x+\underline{\bbR}\cong\underline{\ind D}^\calA_y,
    \end{equation}
as follows.

We define 
    \begin{align}
    \widetilde{x}:\bbR\times S^1&\to T^*Q \\
    (s,t)&\mapsto x(t). \nonumber
    \end{align}
Let $D^\calA_x$ be a \emph{Floer abstract cap} for $x$, i.e., a Fredholm operator (with totally real boundary conditions)
    \begin{equation}
    D^\calA_x:W^{k,p}\big(\bbR_{\geq0}\times S^1;\widetilde{x}^*T(T^*Q),x^*\Lambda_\std\big)\to W^{k-1,p}\big(\bbR_{\geq0}\times S^1;\widetilde{x}^*T(T^*Q)\big)
    \end{equation}
with asymptotic operator $D(\overline{\partial}_{H^\ell,J^\ell})_x$ at $+\infty$. Given $u\in\operatorname{int}\bbF^\ell(y,x)$, we consider the glued together operator
    \begin{equation}
    D^\calA_x\#D(\overline{\partial}_{H^\ell,J^\ell})_u.
    \end{equation}

\begin{lem}\label{lem:symmetrybreaking}
After fixing the line $s\mapsto(s,1)$ on $\bbR\times S^1$, we have a canonical isomorphism of virtual vector spaces
    \begin{equation}
    \ind\big(D^\calA_x\#D(\overline{\partial}_{H^\ell,J^\ell})_u\big)\cong\ind D^\calA_y.
    \end{equation}
\end{lem}

\begin{proof}
We may canonically identify $D^\calA_x\#D(\overline{\partial}_{H^\ell,J^\ell})_u$ with an operator 
    \begin{equation}
    W^{k,p}\big(\bbR_{\geq0}\times S^1;\widetilde{y}^*T(T^*Q),F\big)\to W^{k-1,p}\big(\bbR_{\geq0}\times S^1;\widetilde{y}^*T(T^*Q)\big),
    \end{equation}
with asymptotic operator $D(\overline{\partial}_{H^\ell,J^\ell})_y$ at $+\infty$, where $F\subset\widetilde{y}^*T(T^*Q)$ is some totally real subbundle which admits a non-canonical homotopy to $b^*\Lambda_\std$ through totally real subbundles. In particular, a choice of such a homotopy is equivalent to a choice of isomorphism of real vector spaces 
    \begin{equation}
    T_{x(0)}Q\cong T_{y(0)}Q.
    \end{equation}
We provide such a choice by parallel transporting along $s\mapsto(s,1)$ (utilizing the Levi-Civita connection on $Q$ with respect to some Riemannian metric).
\end{proof}

As an aside, if we were working in the special case that $Q$ is oriented, we could keep the $S^1$-symmetry intact by instead using the following result.

\begin{lem}\label{lem:nonsymmetrybreaking}
Let $E\equiv F\otimes_\bbR\underline{\bbC}\to[0,1]\times S^1$ be a complex vector bundle over the finite cylinder which is the complexification of a totally real subbundle. Consider a Cauchy-Riemann operator (with totally real boundary conditions):
    \begin{equation}
    D:W^{k,p}([0,1]\times S^1;E,F)\to W^{k-1,p}([0,1]\times S^1;E).
    \end{equation}
If $F$ is oriented, then we have a canonical isomorphism of virtual vector spaces 
    \begin{equation}
    \ind D\cong0.
    \end{equation}
\end{lem}

\begin{proof}
Let $\Gamma:S^1\to[0,1]\times S^1$ be the loop covering $\{1/2\}\times S^1$ exactly once. Since $F\vert_\Gamma$ is oriented, it is trivialized, hence we may collapse $\Gamma$ to a point in order to identify $D$ as the gluing of two operators 
    \begin{equation}
    D_j:W^{k,p}(D^2_j;E_j,F_j)\to W^{k-1,p}(D^2_j;E_j),\;\;j=1,2,
    \end{equation}
where $D^2_j\subset\bbC$ is the unit disk. When performing the gluing, we may only glue together sections which agree at the origins of the two disks. In particular, the glued together operator 
    \begin{equation}
    D_1\#D_2:W^{k,p}([0,1]\times S^1;E,F)\to W^{k-1,p}([0,1]\times S^1;E), 
    \end{equation}
which may be canonically deformed to $D$, canonically satisfies 
    \begin{equation}
    \ind D\cong\ind(D_1\#D_2)\cong\ind D_1+\ind D_2-\Delta^\perp\cong F\vert_0+F\vert_0-E\vert_0\cong0,
    \end{equation}
where the third isomorphism uses that (1) the totally real boundary conditions of $D_j$ extends over $D^2_j$ and (2) the short exact sequence 
    \begin{equation}
    0\to E\vert_0\xrightarrow{\Delta} E\vert_0^{\oplus2}\to\Delta^\perp\to0.
    \end{equation}
\end{proof}

Since the identification in Lemma \ref{lem:symmetrybreaking} is canonical, it immediately extends to families, i.e., we obtain standard twisted stable framings 
    \begin{equation}
    T\scrF^\ell(y,x)+\underline{\ind D}^\calA_x+\underline{\bbR}\cong\underline{\ind D}^\calA_y.
    \end{equation}
That this twisted stable framing (1) extends over the Gromov-compactification and (2) has the correct restriction to a codimension 1 boundary stratum follows analogously to the proof of Corollary \ref{cor:extension}.

Consider the $\bbZ$-graded, only supported in degree $-\dim Q$, local system $\eta$ on $\calL Q$ defined via 
    \begin{equation}\label{eqn:localsystem}
    \eta_\Gamma\equiv\sigma^{TQ}\vert_\Gamma\otimes o^{-1}_{TQ}\vert_{\Gamma(0)}\otimes\big(o_{TQ}\vert_{\Gamma(0)}[\dim Q]\big)^{\otimes-w(\Gamma)},\;\;\Gamma\in\calL Q, 
    \end{equation}
where: $\sigma^{TQ}$ is the local system of trivializations of $\Gamma^*\big(TQ\oplus\det_\bbR(TQ)^{\oplus3}\big)$ (i.e., the space of spin structures), $o_{TQ}$ is the orientation local system, $[\dim Q]$ indicates shifting the degree down by $\dim Q$, and $w(\Gamma)$ equals 0 resp. $-1$ if $\Gamma^*TQ$ is orientable resp. non-orientable. 

\begin{lem}\label{lem:localsystemcomputation}
$\bbF^\ell$ lifts to a framed flow category $\bbF^{{\bbS},\ell,\Lambda_\std}$ satisfying 
    \begin{equation}\label{eqn:aux3}
    H_*(\frakF^{\bbS,\ell,\Lambda_\std};\bbZ)\cong HF^{-*}(T^*Q;H^\ell,J^\ell;\eta^{-1}),\;\;\frakF^{\bbS,\ell,\Lambda_\std}\equiv{\rm Flow}^{\rm fr}(\unit,\bbF^{\bbS,\ell,\Lambda_\std}).
    \end{equation}
\end{lem}

\begin{proof}
Any rigid $u$ connecting $x$ to $y$ determines an isomorphism of determinant lines of the form
    \begin{equation}\label{eqn:aux4}
    \operatorname{det}_\bbR D^\calA_x\cong\operatorname{det}_\bbR D^\calA_y.
    \end{equation}
Moreover, under the base-change functor $-\otimes_\bbS H\bbZ:{\rm Sp}\to{\rm Mod}_{H\bbZ}$, we see that $\frakF^{\bbS,\ell,\Lambda_\std}$ is sent to the cochain complex $\mathcal{C}^{-*}$, where
    \begin{equation}
    \mathcal{C}^j\equiv\bigoplus_{x\in\chi(H^\ell),\deg(x)=j}\bbZ\langle x\rangle
    \end{equation}
with codifferential induced in the usual way by \eqref{eqn:aux4}. The key point is that the determinant line of $D^\calA_x$ is precisely 
    \begin{equation}
    \frako_x[w(x)]\otimes\eta_x,
    \end{equation}
where $\frako_x$ is the usual determinant line associated to $x$, cf. \cite[Definition 1.4.19]{Abo14}; we may verify this as follows.

Endow $\bbC $ with a negative cylindrical end and consider an operator 
    \begin{equation}
    D_x:W^{k,p}(\bbC;\bbC^{\dim Q})\to W^{k,p}(\bbC;\bbC^{\dim Q})
    \end{equation}
with asymptotic operator $D(\overline{\partial}_{H^\ell,J^\ell})_x$ at $-\infty$; recall, $\frako_x\equiv\operatorname{det}_\bbR(D_x)$. Now, by considering the glued together operator $D^\calA_x\#D_x$, and utilizing \cite[Lemma 4.3.1]{Abo14}, we see that 
    \begin{align}
    \operatorname{det}_\bbR(D^\calA_x\#D_x)\cong\operatorname{det}_\bbR D^\calA_x\otimes\frako_x\cong \eta^{-1}_x[-w(x)]&\implies \\
    \operatorname{det}_\bbR D^\calA_x\cong\frako_x[w(x)]\otimes\eta_x&.
    \end{align}
 
Also, identifying \eqref{eqn:aux4} with the isomorphism 
    \begin{equation}
    \frako_x[w(x)]\otimes\eta_x\cong\frako_y[w(y)]\otimes\eta_y,
    \end{equation}
induced by a rigid $u$ connecting $x$ to $y$, will again utilize \cite[Lemma 4.3.1]{Abo14}; the lemma follows.
\end{proof}

Moreover, analogously to before, there is a standard way to construct twisted stable framings: 
    \begin{equation}
    T\frakc_{\ell,\ell+1}(x,y)+\underline{\ind D}^\calA_y\cong\underline{\ind D}^\calA_x.
    \end{equation}
In particular, we have constructed a framed flow bimodule 
    \begin{equation}
    \frakc^\bbS_{\ell,\ell+1}:\bbF^{\bbS,\ell,\Lambda_\std}\to\bbF^{\bbS,\ell+1,\Lambda_\std},
    \end{equation}
hence a map
    \begin{equation}
    \frakc^\bbS_{\ell,\ell+1}:\frakF^{\bbS,\ell,\Lambda_\std}\to\frakF^{\bbS,\ell+1,\Lambda_\std}
    \end{equation}
which agrees, on integral homology, with the usual Floer continuation map. 

\begin{prop}
We have that $\frakF^{\bbS,\Lambda_\std}\equiv\varinjlim_\ell\frakF^{\bbS,\ell,\Lambda_\std}$ satisfies
    \begin{equation}
    H_*(\frakF^{\bbS,\Lambda_\std};\bbZ)\cong SH^{-*}(T^*Q;\eta^{-1}),
    \end{equation}
where the colimit is over a family of increasing slope.
\end{prop}

\section{Spectral Viterbo isomorphism}
We now relate string topology and Floer theory; this section follows, and lifts to spectra, \cite[Chapter 4]{Abo14}.

\subsection{Unstructured}\label{subsec:unstructured}
We will begin by first constructing the Viterbo flow bimodule as an unstructured flow biomodule, i.e., defining the relevant moduli spaces. The subsequent subsections will show the aforementioned unstructured flow bimodule constructed admits (1) a (spherical) complex-oriented lift and (2) a framed lift. 

\subsubsection{Viterbo moduli spaces}
Let $(H^\ell,J^\ell)$ be regular Floer data. We choose a non-degenerate $\bbR_{\geq0}$-dependent family of Hamiltonians $H^+_s\in C^\infty(S^1\times T^*Q)$, which are linear at infinity of the same slope as $H^\ell$ and agree with $H^\ell$ near $+\infty$, such that $X_{H^+_0}\vert_Q=0$. For any $x\in\chi(H^\ell)$, we denote by $\scrU_\ell(x)$ the moduli space of smooth half-cylinders $u:\bbR_{\geq0}\times S^1\to T^*Q$ satisfying
    \begin{equation}
    \begin{cases}
    \big(du-dt\otimes X_{H^+_s})^{0,1}=0 \\
    \lim_{s\to+\infty}u(s,t)=x(t) \\
    u(0,t)\in Q.
    \end{cases}
    \end{equation}
We define 
    \begin{align}
    \eval:\scrU_\ell(x)&\to\calL Q \\
    u&\mapsto\big(t\mapsto u(0,t)\big). \nonumber
    \end{align}
By viewing $\scrU_\ell(x)$ as the zero set of a section of an appropriate Banach bundle and investigating the corresponding family of Fredholm operators,
    \begin{equation}
    D^{\scrU_\ell}\equiv\big\{D^{\scrU_\ell}_u\big\}_u,
    \end{equation}
which are surjective for generic data, we see it is a smooth manifold whose tangent bundle is classified by $\ind D^{\scrU_\ell}$.

Now, consider the map 
    \begin{align}
    \eval_r:\scrU_\ell(x)&\to Q^r \\
    u&\mapsto\Big(\eval(u)(0),\ldots,\eval(u)\big((r-1)/r\big)\Big); \nonumber
    \end{align}
this has image contained in $\calL^rQ$ for $r\gg0$. Moreover, we have a homotopy commutative diagram 
    \begin{equation}
    \begin{tikzcd}[column sep=large]
    \scrU_\ell(x)\arrow[d,"\eval"]\arrow[r,"\eval_r"]\arrow[dr,"\eval_{r+1}"] & \calL^rQ\arrow[d,"i_{r,r+1}"] \\ 
    \calL Q & \calL^{r+1}Q\arrow[l,"{\rm geo}"],
    \end{tikzcd}
    \end{equation}
cf. \cite[Exercise 3.2.3]{Abo14}. Finally, $\eval_r$ is, for generic data, transverse to all stable manifolds of $f^r$. In particular, we may define the hybrid moduli space $\scrU_{\ell,r}(x,a)\equiv\eval_r^{-1}\big(W^s(a;f^r)\big)$; this has a natural Gromov-compactification $\bbU_{\ell,r}(x,a)$ which is a compact smooth manifold with corners, stratified by breakings at critical points resp. Hamiltonian orbits, whose codimension 1 boundary strata are enumerated by gluing maps of the form 
    \begin{align}
    \bbF^\ell(x,x')\times\bbU_{\ell,r}(x',a)&\hookrightarrow\bbU_{\ell,r}(x,a), \\
    \bbU_{\ell,r}(x,a')\times\bbM^r(a',a)&\hookrightarrow\bbU_{\ell,r}(x,a).
    \end{align}
(These statements follow by the techniques in \cite[Section 6]{Lar21}; moreover, we have the analogous result to Proposition \ref{prop:indexfloer} via the techniques in \cite[Section 7]{Lar21}.) In particular, we have constructed an unstructured flow biomodule 
    \begin{equation}
    \bbU_{\ell,r}:\bbF^\ell\to\bbM^r.
    \end{equation}
It remains to show compatibility with the continuation resp. inclusion unstructured flow bimodules; the proof of both compatibilities is similar.

\subsubsection{Compatibility with directed systems}
We consider the inclusions. Compare the following with \cite[Part 4.3.4]{Abo14}.

\begin{prop}\label{prop:unstructuredinclusion}
The unstructured flow bimodules $\bbI_{r,r+1}\circ\bbU_{\ell,r}$ and $\bbU_{\ell,r+1}$ are homotopic as unstructured flow bimodules.
\end{prop}

\begin{proof}
The proof of the proposition amounts to constructing an unstructured flow 2-simplex 
    \begin{equation}
    \bbU^\bbI:\bbI_{r,r+1}\circ\bbU_{\ell,r}\Rightarrow\bbU_{\ell,r+1}.
    \end{equation}
Let $\phi^r_s$ denote the flow of $-\grad f^r$ and consider the map
    \begin{align}
    \eval_{\bbR,r}:\scrU^\bbI(x)\equiv\bbR\times\scrU_\ell(x)&\to\calL^{r+1}Q \\
    (s,u)&\mapsto(i_{r,r+1}\circ\phi^r_s\circ\eval_r)(u). \nonumber
    \end{align}
Observe, $\scrU^\bbI(x,a)\equiv\eval_{\bbR,r}^{-1}\big(W^s(a;f^{r+1})\big)$ has a natural Gromov-compactification $\bbU^\bbI(x,a)$ which is a compact smooth manifold with corners, given by allowing breakings at critical points resp. Hamiltonian orbits, whose other two codimension 1 boundary strata are given by the boundary at $s=0$ resp. $s=+\infty$: 
    \begin{equation}
    \bbU_{\ell,r+1}(x,a)\;\;{\rm resp.}\;\;\Bigg(\coprod_{a'\in\crit(f^r)}\bbU_{r,\ell}(x,a')\times\bbI_{r,r+1}(a',a)\Bigg)\Bigg/\sim, 
    \end{equation}
where the equivalence relation identifies the various images of the form
    \begin{equation}
    \begin{tikzcd}[column sep=-10ex]
    & \bbU_{\ell,r}(x,a')\times\bbM^r(a',b)\times\bbI_{r,r+1}(b,a)\arrow[dr]\arrow[dl] & \\
    \bbU_{\ell,r}(x,b)\times\bbI_{r,r+1}(b,a) & & \bbU_{\ell,r}(x,a')\times\bbI_{r,r+1}(a',a).
    \end{tikzcd}
    \end{equation}
(Again, these statements follow by the techniques in \cite[Section 6]{Lar21}; moreover, we have the analogous result to Proposition \ref{prop:indexfloer} via the techniques in \cite[Section 7]{Lar21}.) This completes the construction of the desired unstructured flow 2-simplex.
\end{proof}

We consider the continuations. Compare the following with \cite[Part 4.3.5]{Abo14}.

\begin{prop}\label{prop:unstructuredcontinuation}
The unstructured flow bimodules $\bbU_{\ell+1,r}\circ\frakc_{\ell,\ell+1}$ and $\bbU_{\ell,r}$ are homotopic as unstructured flow bimodules.
\end{prop}

\begin{proof}
Again, the proof of the proposition amounts to constructing an unstructured flow 2-simplex 
    \begin{equation}
    \bbU^\frakc:\bbU_{\ell+1,r}\circ\frakc_{\ell,\ell+1}\Rightarrow\bbU_{\ell,r}.
    \end{equation}
Let $(H^+_s,J^+_s)$ be data used to define the $\scrU_\ell(x)$'s. Moreover, let $(H_s,J_s)$ be monotonic Floer continuation data connecting $(H^{\ell+1},J^{\ell+1})$ to $(H^\ell,J^\ell)$. Now, consider a $\bbR$-dependent family of regular Floer data $(H_{\mu,s},J_{\mu,s})$ on the glued surface $(\bbR_{\geq0}\times S^1)\#(\bbR\times S^1)$ limiting to $(H^+_s,J^+_s)$ at $\mu=-\infty$ resp. the gluing of $(H^+_s,J^+_s)$ to $(H_s,J_s)$ at $\mu=+\infty$. We assume $J_{\mu,s}$ is convex near $S^*Q$; moreover, we assume $H_{\mu,s}$ does not increase the slope in the positive $s$-direction. For any $x\in\chi(H^\ell)$, we denote by $\scrU^\frakc(x)$ the moduli space of tuples $(\mu,u)$, where $u:\bbR\times S^1\to T^*Q$ satisfies 
    \begin{equation}
    \begin{cases}
    \big(du-dt\otimes X_{H_{\mu,s}})^{0,1}=0 \\
    \lim_{s\to+\infty}u(s,t)=x(t) \\
    u(0,t)\in Q.
    \end{cases}
    \end{equation}
For generic data, this is a smooth manifold. There is a natural evaluation map 
    \begin{equation}
    \eval_r:\scrU^\frakc(x)\to\calL^rQ,\;\;r\gg0,
    \end{equation}
and we define the hybrid moduli space $\bbU^\frakc(x,a)$ given by the natural Gromov-compactification of $\scrU^\frakc(x,a)\equiv\eval_r^{-1}\big(W^s(a;f^r)\big)$ defined by allowing breakings at critical points resp. Hamiltonian orbits; this, for generic data, is a compact smooth manifold with corners whose codimension 1 boundary strata are enumerated by (1) gluing maps for Morse resp. Floer breakings and (2) the boundary at $\mu=-\infty$ resp. $\mu=+\infty$:
    \begin{equation}
    \bbU_{\ell,r}(x,a)\;\;{\rm resp.}\;\;\Bigg(\coprod_{x'\in\chi(H^{\ell+1})}\frakc_{\ell,\ell+1}(x,x')\times\bbU^\frakc(x',a)\Bigg)\Bigg/\sim,
    \end{equation}
where the equivalence relation identifies the various images of the form
    \begin{equation}
    \begin{tikzcd}[column sep=-10ex]
    & \frakc_{\ell,\ell+1}(x,x')\times\bbF^{\ell+1}(x',y)\times\bbU^\frakc(y,a)\arrow[dr]\arrow[dl] & \\
   \frakc_{\ell,\ell+1}(x,y)\times\bbU^\frakc(y,a) & & \frakc_{\ell,\ell+1}(x,x')\times\bbU^\frakc(x',a).
    \end{tikzcd}
    \end{equation}
(Again, these statements follow by the techniques in \cite[Section 6]{Lar21}; moreover, we have the analogous result to Proposition \ref{prop:indexfloer} via the techniques in \cite[Section 7]{Lar21}) This completes the construction of the desired unstructured flow 2-simplex.
\end{proof}

\subsection{With coefficients in complex bordism}
We will now construct the complex-oriented lift of Subsection \ref{subsec:unstructured}.

\subsubsection{Construction of the ${\rm MU}$-local system}
In this part, we will construct the MU-local system mentioned in Theorem \ref{thm:mainMU}. We may assume that, given any $r$ and $a\in\crit(f^r)$, our Hermitian connection $\nabla$ on $T(T^*Q)$ is flat in a neighborhood of the image of ${\rm geo}(a)\in\calL Q$. The construction will proceed in stages, as follows.

First, consider any $a\in\crit(f^r)$. We define 
    \begin{align}
    \widetilde{a}:\bbR_{\leq0}\times S^1&\to T^*Q \\
    (s,t)&\mapsto{\rm geo}(a)(t). \nonumber
    \end{align}
By assumption, $\widetilde{a}^*T(T^*Q)$ is equipped with a complex trivialization at $-\infty$ via $\widetilde{a}^*\nabla$; hence, we may choose a Cauchy-Riemann operator
    \begin{equation}
    D^\calV_a:W^{k,p}\big(\bbR_{\leq0}\times S^1;\widetilde{a}^*T(T^*Q),a^*TQ\big)\to W^{k-1,p}\big(\bbR_{\leq0}\times S^1;\widetilde{a}^*T(T^*Q)\big)
    \end{equation}
which, over the negative cylindrical end, is standard.\footnote{Technically, in order to deal with 0 being an eigenvalue of the asymptotic operator on the negative cylindrical end, we should work with weighted Sobolev spaces so that $D^\calV_a$ is actually Fredholm; we will ignore this point for brevity.} We define
    \begin{equation}
    \frakV_a\equiv(\ind D^\calV_a)^\bbS\otimes_\bbS{\rm MU}\in{\rm Mod}_{\rm MU};
    \end{equation}
observe, $\frakV_a$ is homotopy equivalent to MU --- but, not canonically so.

Second, consider any $\alpha\in\calL^rQ$. By construction, there exists a unique $a\in\crit(f^r)$ such that $\alpha\in W^s(a;f^r)$, and we define 
    \begin{equation}
    \frakV_\alpha\equiv\frakV_a.
    \end{equation}
We denote by $\widetilde{\Gamma}_{\alpha a}:\bbR_{\geq0}\to\calL^rQ$ the negative gradient trajectory satisfying 
    \begin{equation}
    \widetilde{\Gamma}_{\alpha a}(0)=\alpha,\;\;\lim_{s\to+\infty}\widetilde{\Gamma}_{\alpha a}(s)=a;
    \end{equation}
moreover, we denote by $\widetilde{\Gamma}_{a\alpha}:\bbR_{\leq0}\to\calL Q$ the path in $\calL^rQ$ satisfying 
    \begin{equation}
    \widetilde{\Gamma}_{a\alpha}(s)\equiv\widetilde{\Gamma}_{\alpha a}(-s);
    \end{equation}
finally, we define 
    \begin{equation}
    \Gamma_{\alpha a}\equiv{\rm geo}\circ\widetilde{\Gamma}_{\alpha a}\;\;{\rm and}\;\;\Gamma_{a\alpha}\equiv{\rm geo}\circ\widetilde{\Gamma}_{a\alpha}.
    \end{equation}

Third, we build an $\infty$-functor 
    \begin{equation}\label{eqn:infinityfunctoraux}
    \calV_r:\calL^rQ\to{\rm BGL}_1({\rm MU}), 
    \end{equation}
i.e., an MU-local system on $\calL^rQ$. Recall, ${\rm BGL}_1({\rm MU})$ may be viewed as the subcategory of ${\rm Mod}_{\rm MU}$ whose: 0-simplices are all MU-modules abstractly homotopy equivalent to MU, 1-simplices are homotopy equivalences, 2-simplices are homotopies between homotopy equivalences, etc. In particular, we must describe how an $n$-simplex in $\calL^rQ$ is mapped to an $n$-simplex in ${\rm BGL}_1({\rm MU})$ compatibly with face and degeneracy maps. 

\begin{itemize}
\item Let $\sigma:\Delta^0\to\calL^rQ$ be a 0-simplex; we define 
    \begin{equation}
    \calV_r\vert_{\sigma(0)}\equiv\frakV_{\sigma(0)}.
    \end{equation}
\item Let $\widetilde{\sigma}:\Delta^1\to\calL^rQ$ be a 1-simplex and consider the concatenated path 
    \begin{equation}
    \Gamma\equiv\Gamma_{b\sigma(1)}\#\sigma\#\Gamma_{\sigma(0)a}:\bbR\to\calL Q, 
    \end{equation}
where $\widetilde{\sigma}(0)\in W^s(a;f^r)$, $\widetilde{\sigma}(1)\in W^s(b;f^r)$, and $\sigma(s)\equiv({\rm geo}\circ\widetilde{\sigma})(-s)$. We may choose a (complex linear) Cauchy-Riemann operator
    \begin{equation}
    D^\calV_\Gamma:W^{k,p}\big(\bbR\times S^1;\Gamma^*T(T^*Q)\big)\to W^{k-1,p}\big(\bbR\times S^1;\Gamma^*T(T^*Q)\big)
    \end{equation}
which over the negative resp. positive cylindrical end agrees with the restriction of $D^\calV_b$ resp. $D^\calV_a$ to the negative cylindrical end. (We would like to emphasize that such a choice lies in a contractible space of choices.) Utilizing essentially the same proof as that of Lemma \ref{lem:symmetrybreaking}, we have a canonical isomorphism of virtual vector spaces 
    \begin{equation}
    \ind(D^\calV_\Gamma\#D^\calV_a)\cong\ind D^\calV_b.
    \end{equation}
We define 
    \begin{equation}
    \calV_r(\widetilde{\sigma}):\calV_r\vert_{\widetilde{\sigma}(0)}\xrightarrow{\sim}\calV_r\vert_{\widetilde{\sigma}(1)}
    \end{equation}
as follows. By using the fact that we have a canonical homotopy equivalence
    \begin{equation}
    (\ind D^\calV_\Gamma)^\bbS\otimes_\bbS{\rm MU}\simeq{\rm MU},
    \end{equation}
since $D^\calV_\Gamma$ is a complex linear operator and ${\rm BU}\to{\rm BGL}_1({\rm MU})$ is canonically null-homotopic, we immediately obtain a canonical homotopy equivalence
    \begin{equation}
    \calV_r\vert_{\widetilde{\sigma}(0)}=\frakV_{\widetilde{\sigma}(0)}\xrightarrow{\sim}\frakV_{\widetilde{\sigma}(1)}=\calV_r\vert_{\widetilde{\sigma}(1)}.
    \end{equation}

\item Since, for any 1-simplex $\widetilde{\sigma}:\Delta^1\to\calL^rQ$, we constructed a canonical homotopy equivalence $\calV_r\vert_{\widetilde{\sigma}(0)}\simeq\calV_r\vert_{\widetilde{\sigma}(1)}$, this homotopy equivalence immediately extends to families of operators, i.e., we may inductively construct our desired $\infty$-functor \eqref{eqn:infinityfunctoraux}.
\end{itemize}

Fourth, we build an MU-local system on $\calL Q$.

\begin{prop}
We have the following commutative diagram of $\infty$-functors:
    \begin{equation}
    \begin{tikzcd}
    \calL^{r+1}Q\arrow[dr,"\calV_{r+1}"] & \\
    \calL^rQ\arrow[u,"i_{r,r+1}"]\arrow[r,"\calV_r"] & {\rm BGL}_1({\rm MU}).
    \end{tikzcd}
    \end{equation}
In particular, we may define 
    \begin{equation}
    \calV\equiv\varinjlim_r\calV_r:\calL Q\to{\rm BGL}_1({\rm MU}).
    \end{equation}
\end{prop}

\begin{proof}
For every $\alpha\in\calL^rQ$, we will construct a canonical homotopy equivalence 
    \begin{equation}
    \calV_r\vert_\alpha\simeq i_{r,r+1}^*\calV_{r+1}\vert_\alpha;
    \end{equation}
then, since the homotopy equivalence is canonical, it will immediately extend to families, whence the proposition.

By construction, there exists a unique $a\in\crit(f^r)$ and $b\in\crit(f^{r+1})$ such that $\alpha\in W^s(a;f^r)$ and $i_{r,r+1}(\alpha)\in W^s(b;f^{r+1})$. Consider the concatenated path 
    \begin{equation}
    \Gamma\equiv\Gamma_{bi_{r,r+1}(\alpha)}\#\Gamma_{\alpha a}.
    \end{equation}
Again, utilizing essentially the same proof as that of Lemma \ref{lem:symmetrybreaking}, we obtain a canonical identification
    \begin{equation}
    \ind(D^\calV_\Gamma\#D^\calV_a)\cong\ind D^\calV_b,
    \end{equation}
and thus a canonical homotopy equivalence 
    \begin{equation}
    \calV_r\vert_\alpha=\frakV_\alpha\xrightarrow{\sim}\frakV_{i_{r,r+1}(\alpha)}=i_{r,r+1}^*\calV_{r,r+1}\vert_a,
    \end{equation}
as desired.
\end{proof}

Finally, we build our desired MU-local system on $\calL Q$. Let
    \begin{equation}
    \epsilon:\calL Q\to{\rm BGL}_1({\rm MU})
    \end{equation}
be the local system defined as follows. We have a decomposition 
    \begin{equation}
    \calL Q=\calL_0\amalg\calL_1,
    \end{equation}
where $\calL_0$ resp. $\calL_1$ consists of loops $\Gamma$ such that $\Gamma^*TQ$ is orientable resp. non-orientable. On $\calL_0$, we define $\epsilon$ to be the trivial MU-local system $\underline{\rm MU}$. On $\calL_1$, we define $\epsilon$ to be the trivial MU-local system $\underline{\Sigma^{-1}{\rm MU}}$. It is the MU-local system 
    \begin{equation}
    \Sigma^{-\dim Q}\calV\otimes\epsilon:\calL Q\to{\rm BGL}_1({\rm MU})
    \end{equation}
we are ultimately interested in.

\subsubsection{Viterbo moduli spaces: complex-oriented}
As before, we consider the moduli spaces given by the $\bbU_{\ell,r}(x,a)$'s. We will produce a twisted spherical complex orientation of $\bbU_{\ell,r}(x,a)$, as follows.

\begin{lem}
We have a canonical isomorphism of ${\rm MU}$-local systems
    \begin{equation}
    \Big(\big(T\scrU_\ell(x)\big)^\bbS\otimes_\bbS{\rm MU}\Big)\otimes_{\rm MU}\eval_r^*\Sigma^{\deg(x)-\dim Q-w(x)}\calV_r\cong\underline{{\rm MU}}.
    \end{equation}
\end{lem}

\begin{proof}
We will construct the following canonical pointwise homotopy equivalence:
    \begin{equation}
    \Big(\big(T_u\scrU_\ell(x)\big)^\bbS\otimes_\bbS{\rm MU}\Big)\otimes_{\rm MU}\eval_r^*\Sigma^{\deg(x)-\dim Q-w(x)}\calV_r\vert_u\simeq{\rm MU}.
    \end{equation}
Consider the glued together operator $D^\calV_{\eval_r(u)}\#D^{\scrU_\ell}_u$. Analogously to the discussion surrounding \eqref{eqn:analogouscomplex}, we have a canonical identification 
    \begin{equation}
    \ind\Big(D^\calV_{\eval_r(u)}\#D^{\scrU_\ell}_u\oplus\overline{D^\calV_{\eval_r(u)}\#D^{\scrU_\ell}_u}\Big)\cong\ind(D^\calV_{\eval_r(u)}\#D^{\scrU_\ell}_u)+\bbR^{\deg(x)-\dim Q-w(x)},
    \end{equation}
where the left hand side is a complex virtual vector space. Now, the proposition follows after taking the associated MU-module of the previous equation.
\end{proof}

\begin{cor}\label{cor:extension}
There exists a virtual bundle $I(x,a)\to\scrU(x,a)$, which is trivialized as an ${\rm MU}$-local system and satisfies the natural associativity conditions, such that we have the following canonical isomorphism of virtual bundles:
    \begin{equation}
    T\scrU_\ell(x,a)+\underline{\bbR}^{I(a)}+\underline{(\Sigma^{-\dim Q}\calV_r\otimes\epsilon)}\vert_a\cong I(x,a)+\underline{\bbR}^{-\deg(x)}.
    \end{equation}
Moreover, we have that this twisted spherical complex orientation extends over the Gromov-compactification and, when restricted to a codimension 1 boundary stratum, agrees with the already constructed twisted spherical complex orientation on that codimension 1 boundary stratum.
\end{cor}

\begin{proof}
By splitting the short exact sequence 
    \begin{equation}
    0\to T\scrU_\ell(x,a)\to T\scrU_\ell(x)\to TW^u(a;f^r)\to0, 
    \end{equation}
we obtain an identification 
    \begin{equation}
    T\scrU_\ell(x,a)+\underline{\bbR}^{I(a)}\cong T\scrU_\ell(x).
    \end{equation}
Now, the previous lemma provides the existence of $I(x,a)\to\scrU_\ell(x,a)$, satisfying the conditions in the statement of this corollary, such that 
    \begin{equation}
    T\scrU_\ell(x)\vert_{\scrU_\ell(x,a)}+\underline{(\Sigma^{\deg(x)-\dim Q}\calV_r\otimes\epsilon)}\vert_a\cong I(x,a);
    \end{equation}
hence,
    \begin{align}
    T\scrU_\ell(x,a)+\underline{\bbR}^{I(a)}+\underline{(\Sigma^{\deg(x)-\dim Q}\calV_r\otimes\epsilon)}\vert_a\cong I(x,a)&\implies \\
    T\scrU_{\ell,r}(x,a)+\underline{\bbR}^{I(a)}+\underline{(\Sigma^{-\dim Q}\calV_r\otimes\epsilon)}\vert_a\cong I(x,a)+\underline{\bbR}^{-\deg(x)}&.
    \end{align}

The claims about this twisted spherical complex orientation (1) extending over the Gromov-compactification and (2) having the correct restriction on a codimension 1 boundary stratum both follow in a straightforward manner since our various twisted spherical complex orientations are constructed via gluing operators. Since both previous and future arguments use similar ideas, we will explain the proof here and skip the details in other places for brevity.

First, our twisted spherical complex orientation extends over the Gromov-compactification by the analogue of Proposition \ref{prop:indexfloer}. 

Second, we consider a codimension 1 boundary stratum of the form
    \begin{equation}
    \bbF^\ell(x,x')\times\bbU_{\ell,r}(x',a)\hookrightarrow\bbU_{\ell,r}(x,a).
    \end{equation}
Let 
    \begin{equation}
    (u_1,u_2)\in\operatorname{int}\big(\bbF^\ell(x,x')\times\bbU_{\ell,r}(x',a)\big).
    \end{equation}
The prescribed twisted spherical complex orientation at $(u_1,u_2)$ is obtained by considering the two operators    
    \begin{align}
    D(\overline{\partial}_{H^\ell,J^\ell})_{u_1}\oplus \overline{D(\overline{\partial}_{H^\ell,J^\ell})_{u_1}}, \\
    D^\calV_a\#D^{\scrU_\ell}_{u_2}\oplus\overline{D^\calV_a\#D^{\scrU_\ell}_{u_2}}
    \end{align}
and summing their index bundles with $TW^u(a;f^r)\cong\underline{\bbR}^{I(a)}$. The claim follows by observing that the glued together operator
    \begin{equation}
    \Big(D^\calV_a\#D^{\scrU_\ell}_{u_2}\oplus\overline{D^\calV_a\#D^{\scrU_\ell}_{u_2}}\Big)\#\Big(D(\overline{\partial}_{H^\ell,J^\ell})_{u_1}\oplus \overline{D(\overline{\partial}_{H^\ell,J^\ell})_{u_2}}\Big)
    \end{equation}
can be canonically deformed to the operator 
    \begin{equation}
    D^\calV_a\# D^{\scrU_\ell}_u\oplus\overline{D^\calV_a\# D^{\scrU_\ell}_u}, 
    \end{equation}
where $u\equiv u_1\#u_2$ is the gluing of $u_1$ and $u_2$.

Finally, we consider a codimension 1 boundary stratum of the form
    \begin{equation}
    \bbU_{\ell,r}(x,a')\times\bbM^r(a',a)\hookrightarrow\bbU_{\ell,r}(x,a).
    \end{equation}
Let 
    \begin{equation}
    (u,\gamma)\in\operatorname{int}\big(\bbU_{\ell,r}(x,a')\times\bbM^r(a',a)\big).
    \end{equation}
The prescribed twisted spherical complex orientation at $(u,\gamma)$ is obtained by: considering the two operators 
    \begin{align}
    D^\calV_{a'}\# D^{\scrU_\ell}_u\oplus\overline{D^\calV_{a'}\# D^{\scrU_\ell}_u}, \\
    D(\Xi_r)_\gamma;
    \end{align}
summing the index bundle of the first operator with $TW^u(a';f^r)\cong\underline{\bbR}^{I(a')}$; summing the index bundle of the second operator with 
    \begin{equation}
    TW^u(a;f^r)+\underline{(\Sigma^{-\dim Q}\calV_r\otimes\epsilon)}\vert_a-TW^u(a';f^r)-\underline{(\Sigma^{-\dim Q}\calV_r\otimes\epsilon)}\vert_{a'};
    \end{equation}
and taking one final sum. The claim follows by observing the cancellations of the $TW^u(a';f^r)$ and $\underline{(\Sigma^{-\dim Q}\calV_r\otimes\epsilon)}\vert_{a'}$ terms in the final sum.
\end{proof}

In particular, we have constructed a spherically complex-oriented flow bimodule 
    \begin{equation}
    \bbU^{\rm scx}_{\ell,r}:\bbF^{{\rm scx},\ell}\to\bbM^{{\rm scx},r,\Sigma^{-\dim Q}\calV\otimes\epsilon}.
    \end{equation}
It remains to show compatibility with the continuation resp. inclusion spherically complex-oriented flow bimodules.

\subsubsection{Directed systems: complex-oriented}
Again, we first consider the inclusions. 

\begin{prop}
The spherically complex-oriented flow bimodules $\bbI^{\rm scx}_{r,r+1}\circ\bbU^{\rm scx}_{\ell,r}$ and $\bbU^{\rm scx}_{\ell,r+1}$ are homotopic as spherically complex-oriented flow bimodules.
\end{prop}

\begin{proof}
The key point is to provide spherical complex orientations on the moduli spaces $\bbU^\bbI(x,a)$ appearing in the proof of Proposition \ref{prop:unstructuredinclusion} which, when restricted to codimension 1 boundary strata, agree with the already constructed spherical complex orientations.

By considering $u\in\scrU^\bbI(x)$ together with the operator 
    \begin{equation}
    D^\calV_{\eval_{\bbR,r}(u)}\#D^{\scrU_\ell}_u\oplus\overline{D^\calV_{\eval_{\bbR,r}}(u)\#D^{\scrU_\ell}_u},
    \end{equation}
and using the fact that 
    \begin{equation}
    T\scrU^\bbI(x)\cong\underline{\bbR}+T\scrU_\ell(x),
    \end{equation}
we obtain the following canonical isomorphism of MU-local systems:
    \begin{equation}\label{eqn:aux5}
    \Big(\big(T\scrU^\bbI(x)-\underline{\bbR}\big)^\bbS\otimes_\bbS{\rm MU}\Big)\otimes_{\rm MU}\eval_{\bbR,r}^*\Sigma^{\deg(x)-\dim Q-w(x)}\calV_{r+1}\cong\underline{{\rm MU}}.
    \end{equation}

By splitting the short exact sequence 
    \begin{equation}
    0\to T\scrU^\bbI(x,a)\to T\scrU^\bbI(x)\to TW^u(a;f^{r+1})\to0, 
    \end{equation}
we obtain an identification 
    \begin{equation}
    T\scrU^\bbI(x,a)+\underline{\bbR}^{I(a)}\cong T\scrU^\bbI(x).
    \end{equation}
Now, \eqref{eqn:aux5} provides the existence of $I(x,a)\to\scrU^\bbI(x,a)$, satisfying the appropriate MU-triviality and associativity conditions, such that 
    \begin{equation}
    T\scrU^\bbI(x)\vert_{\scrU^\bbI(x,a)}-\underline{\bbR}+\underline{(\Sigma^{\deg(x)-\dim Q}\calV_{r+1}\otimes\epsilon)}\vert_a\cong I(x,a);
    \end{equation}
hence,
    \begin{equation}
    T\scrU^\bbI(x,a)+\underline{\bbR}^{I(a)}+\underline{(\Sigma^{-\dim Q}\calV_{r+1}\otimes\epsilon)}\vert_a\cong I(x,a)+\underline{\bbR}+\underline{\bbR}^{-\deg(x)}.
    \end{equation}

The claims about this spherical complex orientation (1) extending over the Gromov-compactification and (2) having the correct restriction on a codimension 1 boundary stratum both follow analogously to the proof of Corollary \ref{cor:extension}; the proposition follows.
\end{proof}

We continue with the continuations.

\begin{prop}
The spherically complex-oriented flow bimodules $\bbU^{\rm scx}_{\ell+1,r}\circ\frakc^{\rm scx}_{\ell,\ell+1}$ and $\bbU^{\rm scx}_{\ell,r}$ are homotopic as spherically complex-oriented flow bimodules.
\end{prop}

\begin{proof}
Again, the key point is to provide spherical complex orientations on the moduli spaces $\bbU^\frakc(x,a)$ appearing in the proof of Proposition \ref{prop:unstructuredcontinuation} which, when restricted to codimension 1 boundary strata, agree with the already constructed spherical complex orientations.

By considering $u\in\scrU^\frakc(x)$ together with the operator 
    \begin{equation}
    D^\calV_{\eval_r(u)}\#D^{\scrU_\ell}_u\oplus\overline{D^\calV_{\eval_r(u)}\#D^{\scrU_\ell}_u},
    \end{equation}
and using the fact that 
    \begin{equation}
    T\scrU^\frakc(x)\cong T\scrU_\ell(x)+\underline{\bbR},
    \end{equation}
we obtain the following canonical isomorphism of MU-local systems:
    \begin{equation}\label{eqn:aux2}
    \Big(\big(T\scrU^\frakc(x)-\underline{\bbR}\big)^\bbS\otimes_\bbS{\rm MU}\Big)\otimes_{\rm MU}\eval_r^*\Sigma^{\deg(x)-\dim Q-w(x)}\calV_r\cong\underline{{\rm MU}}.
    \end{equation}

By splitting the short exact sequence 
    \begin{equation}
    0\to T\scrU^\frakc(x,a)\to T\scrU^\frakc(x)\to TW^u(a;f^r)\to0, 
    \end{equation}
we obtain an identification 
    \begin{equation}
    T\scrU^\frakc(x,a)+\underline{\bbR}^{I(a)}\cong T\scrU^\frakc(x).
    \end{equation}
Now, \eqref{eqn:aux2} provides the existence of $I(x,a)\to\scrU^\frakc(x,a)$, satisfying the appropriate MU-triviality and associativity conditions, such that 
    \begin{equation}
    T\scrU^\frakc(x)\vert_{\scrU^\frakc(x)}-\underline{\bbR}+\underline{(\Sigma^{\deg(x)-\dim Q}\calV_r\otimes\epsilon)}\vert_a\cong I(x,a);
    \end{equation}
hence,
    \begin{equation}
    T\scrU^\frakc(x,a)+\underline{\bbR}^{I(a)}+\underline{(\Sigma^{-\dim Q}\calV_r\otimes\epsilon)}\vert_a\cong I(x,a)+\underline{\bbR}+\underline{\bbR}^{-\deg(x)}.
    \end{equation}

The claims about this spherical complex orientation (1) extending over the Gromov-compactification and (2) having the correct restriction on a codimension 1 boundary stratum both follow analogously to the proof of Corollary \ref{cor:extension}; the proposition follows.
\end{proof}

Appropriately modifying the proof of Lemma \ref{lem:localsystemcomputation} yields the following result.

\begin{lem}
We have that 
    \begin{equation}
    H_*\big((\calL^rQ)^{\Sigma^{-\dim Q}\calV\otimes\epsilon};\bbZ\big)\cong H_*(\calL^rQ;\eta).
    \end{equation}
\end{lem}

We have now all of the ingredients necessary to prove Theorem \ref{thm:mainMU}. 

\begin{proof}[Proof of Theorem \ref{thm:mainMU}]
We have constructed a map 
    \begin{equation}
    \bbU^{\rm scx}_{\ell,r}:\frakF^{{\rm scx},\ell}\to\frakM^{{\rm scx},r,\Sigma^{-\dim Q}\calV\otimes\epsilon}\simeq(\calL^rQ)^{\Sigma^{-\dim Q}\calV\otimes\epsilon}\otimes_{\rm MU}\frakR^{\rm scx}
    \end{equation}
which is compatible with the inclusions and continuations; hence, we induce a map 
    \begin{equation}\label{eqn:viterbo}
    \bbU^{\rm scx}:\frakF^{\rm MU}\otimes_{\rm MU}\frakR^{\rm scx}\to(\calL Q)^{\Sigma^{-\dim Q}\calV\otimes\epsilon}\otimes_{\rm MU}\frakR^{\rm scx}
    \end{equation}
which, on integral homology, is precisely the Viterbo isomorphism of \cite[Theorem 5.1.1]{Abo14}. Since every spectrum in sight is bounded below, a spectral Whitehead theorem for $H\bbZ$-module spectra shows the previous map is in fact a homotopy equivalence. 

Now, the results in \cite[Section 7]{PS25b} show there exists a complex-oriented flow category $\bbX^r$ satisfying 
    \begin{equation}
    \bbX^r\otimes_{\rm MU}\frakR^{\rm sfr}\cong\bbM^{{\rm scx},r,\Sigma^{-\dim Q}\calV\otimes\epsilon}\;\;{\rm and}\;\;\flow^{\rm cx}(\unit,\bbX^r)\simeq(\calL^rQ)^{\Sigma^{-\dim Q}\calV\otimes\epsilon}.
    \end{equation}
Essentially, we may upgrade the twisted spherical complex orientations on the Morse moduli spaces to actual stable complex structures, at the cost of altering the topology of the moduli spaces themselves. Moreover, the results in \emph{loc. cit.} show there exists a complex-oriented flow bimodule 
    \begin{equation}
    \bbU^{\rm MU}_{\ell,r}:\bbF^{\rm MU,\ell}\to\bbX^r
    \end{equation}
satisfying 
    \begin{equation}
    \bbU^{\rm MU}_{\ell,r}\otimes_{\rm MU}\frakR^{\rm scx}\cong\bbU^{\rm scx}_{\ell,r};
    \end{equation}
we may choose these to be compatible with inclusions and continuations. Therefore, our previous work, combined with Lemma \ref{lem:conservative}, implies we have a homotopy equivalence
    \begin{equation}
    \frakF^{\rm MU}\simeq(\calL Q)^{\Sigma^{-\dim Q}\calV\otimes\epsilon},
    \end{equation}
as desired.
\end{proof}

\subsection{With coefficients in framed bordism}
We end by constructing the framed lift of Subsection \ref{subsec:unstructured}.

\subsubsection{Viterbo moduli spaces: framed}
As before, we consider the moduli spaces given by the $\bbU_{\ell,r}(x,a)$'s. We will produce a twisted stable framing of $\bbU_{\ell,r}(x,a)$, as follows. An appropriate modification of the proof of Lemma \ref{lem:symmetrybreaking} yields the following. 

\begin{lem}\label{lem:symmetrybreaking2}
We have the following canonical isomorphism of virtual bundles: 
    \begin{equation}
    T\scrU_\ell(x)\cong\underline{\ind D}^\calA_x.
    \end{equation}
\end{lem}

\begin{cor}
We have the following canonical isomorphism of virtual bundles:
    \begin{equation}
    T\scrU_\ell(x,a)+\underline{\bbR}^{I(a)}\cong\underline{\ind D}^\calA_x.
    \end{equation}
Moreover, we have that this twisted stable framing extends over the Gromov-compactification and, when restricted to a codimension 1 boundary stratum, agrees with the already constructed twisted stable framing on that codimension 1 boundary stratum.
\end{cor}

\begin{proof}
By splitting the short exact sequence 
    \begin{equation}
    0\to T\scrU_\ell(x,a)\to T\scrU_\ell(x)\to TW^u(a;f^r)\to0, 
    \end{equation}
we obtain an identification 
    \begin{equation}
    T\scrU_\ell(x,a)+\underline{\bbR}^{I(a)}\cong T\scrU_\ell(x)\cong\underline{\ind D}^\calA_x.
    \end{equation}
    
The claims about this twisted stable framing (1) extending over the Gromov-compactification and (2) having the correct restriction on a codimension 1 boundary stratum both follow analogously to the proof of Corollary \ref{cor:extension}, whence the corollary.
\end{proof}

In particular, we have constructed a framed flow bimodule
    \begin{equation}
    \bbU^{ \bbS}_{\ell,r}:\bbF^{\bbS,\ell,\Lambda_\std}\to\bbM^{\bbS,r}.
    \end{equation}
It remains to show compatibility with the continuation resp. inclusion framed flow bimodules.

\subsubsection{Directed systems: framed}
We consider the inclusions. 

\begin{prop}
The framed flow bimodules $\bbI^\bbS_{r,r+1}\circ\bbU^\bbS_{\ell,r}$ and $\bbU^\bbS_{\ell,r+1}$ are homotopic as framed flow bimodules.
\end{prop}

\begin{proof}
We need to provide twisted stable framings on the moduli spaces $\bbU^\bbI(x,a)$ appearing in the proof of Proposition \ref{prop:unstructuredinclusion} which, when restricted to codimension 1 boundary strata, agree with the already constructed twisted stable framings.

Lemma \ref{lem:symmetrybreaking2}, combined with the fact that
    \begin{equation}
    T\scrU^\bbI(x)\cong\underline{\bbR}+T\scrU_\ell(x),
    \end{equation}
yields the following canonical isomorphism of virtual bundles: 
    \begin{equation}
    T\scrU^\bbI(x)\cong\underline{\bbR}+\underline{\ind D}^\calA_x.
    \end{equation}

By splitting the short exact sequence 
    \begin{equation}
    0\to T\scrU^\bbI(x,a)\to T\scrU^\bbI(x)\to TW^u(a;f^{r+1})\to0, 
    \end{equation}
we obtain an identification 
    \begin{equation}
    T\scrU^\bbI(x,a)+\underline{\bbR}^{I(a)}\cong T\scrU^\bbI(x)\cong\underline{\bbR}+\underline{\ind D}^\calA_x.
    \end{equation}

That this twisted stable framing (1) extends over the Gromov-compactification and (2) has the correct restriction on a codimension 1 boundary stratum both follow analogously to the proof of Corollary \ref{cor:extension}; the proposition follows.
\end{proof}

We end with the continuations. 

\begin{prop}
The framed flow bimodules $\bbU^\bbS_{\ell+1,r}\circ\frakc^\bbS_{\ell,\ell+1}$ and $\bbU^\bbS_{\ell,r}$ are homotopic as framed flow bimodules.
\end{prop}

\begin{proof}
We need to provide twisted stable framings on the moduli spaces $\bbU^\frakc(x,a)$ appearing in the proof of Proposition \ref{prop:unstructuredcontinuation} which, when restricted to codimension 1 boundary strata, agree with the already constructed twisted stable framings.

Lemma \ref{lem:symmetrybreaking2}, combined with the fact that
    \begin{equation}
    T\scrU^\frakc(x)\cong\underline{\bbR}+T\scrU^\frakc(x),
    \end{equation}
yields the following canonical isomorphism of virtual bundles: 
    \begin{equation}
    T\scrU^\frakc(x)\cong\underline{\bbR}+\underline{\ind D}^\calA_x.
    \end{equation}

By splitting the short exact sequence 
    \begin{equation}
    0\to T\scrU^\frakc(x,a)\to T\scrU^\frakc(x)\to TW^u(a;f^r)\to0, 
    \end{equation}
we obtain an identification 
    \begin{equation}
    T\scrU^\frakc(x,a)+\underline{\bbR}^{I(a)}\cong T\scrU^\frakc(x)\cong\underline{\bbR}+\underline{\ind D}^\calA_x.
    \end{equation}

That this twisted stable framing (1) extends over the Gromov-compactification and (2) has the correct restriction on a codimension 1 boundary stratum both follow analogously to the proof of Corollary \ref{cor:extension}; the proposition follows.
\end{proof}

We have now all of the ingredients necessary to prove Theorem \ref{thm:mainS}. 

\begin{proof}[Proof of Theorem \ref{thm:mainS}]
We have constructed a map 
    \begin{equation}
    \bbU^{ \bbS}_{\ell,r}:\frakF^{\bbS,\ell,\Lambda_\std}\to\frakM^{\bbS,r}\simeq\Sigma^\infty_+\calL^rQ
    \end{equation}
which is compatible with the inclusions and continuations; hence, we induce a map 
    \begin{equation}
    \bbU^{ \bbS}:\frakF^{\bbS,\Lambda_\std}\to\Sigma^\infty_+\calL Q
    \end{equation}
which, on integral homology, is precisely the Viterbo isomorphism of \cite[Theorem 5.1.1]{Abo14}. Since every spectrum in sight is bounded below, a spectral Whitehead theorem for $H\bbZ$-module spectra shows the previous map is in fact a homotopy equivalence, as desired.
\end{proof}

\bibliography{References}{}
\bibliographystyle{alpha.bst}
\end{document}